\RequirePackage{rotating}
\documentclass[12pt]{article}

\usepackage[margin=1in]{geometry}

\usepackage{amsmath}
\usepackage{amssymb}
\usepackage{amsthm}
\usepackage{url}
\usepackage{graphicx}

\usepackage{enumitem}

\usepackage{tikz} 
\usetikzlibrary{automata,arrows,shapes,topaths,trees,backgrounds,decorations.pathreplacing,positioning}

\usepackage{hyperref}

\usepackage[all,cmtip]{xy}
\usepackage{varwidth}

\newtheorem{theorem}{Theorem}[section]
\newtheorem{lemma}[theorem]{Lemma}

\theoremstyle{definition}
\newtheorem{definition}[theorem]{Definition}

\newtheorem{question}[theorem]{Question}
\newtheorem{notation}[theorem]{Notation}

\numberwithin{subcase}{cases}

\theoremstyle{remark}
\newtheorem{remark}[theorem]{Remark}

\newcommand{\mc}[1]{\mathcal{#1}}

\newcommand{\la}{\langle}
\newcommand{\ra}{\rangle}

\usepackage{todonotes}
\newcommand{\seq}{\mathrm{seq}_k(\mc{A})}

\newcommand{\Qp}{\mathbb{Z}_{(p)}}
\newcommand{\Q}[1]{\mathbb{Z}_{(#1)}}
\newcommand{\nBp}{\mathbb{Z}_{(p)}\mc{A}}
\newcommand{\Bp}{G^g_{\divs p}}
\newcommand{\tBp}{G^h_{\divs p}}
\newcommand{\Es}[2]{\mathbb{E}_{#2}^{#1}}
\newcommand{\Eq}{\mathrm{Equiv}}
\newcommand{\Znz}{\mathbb{Z}_{\neq 0}}
\newcommand{\Qnz}{\mathbb{Q}_{\neq 0}}
\DeclareMathOperator{\divs}{div}
\newcommand{\n}{\mathcal{N}}
\newcommand{\m}{\mathcal{M}}
\newcommand{\rE}{\mathrel{E}}
\DeclareMathOperator{\Mod}{Mod}

\DeclareMathOperator{\Cl}{Cl}

\newcommand{\N}{\mathbb{N}}

\DeclareMathOperator{\tp}{tp}

\newcommand{\bigdoublewedge}{%
  \mathop{
    \mathchoice{\bigwedge\mkern-15mu\bigwedge}
               {\bigwedge\mkern-12.5mu\bigwedge}
               {\bigwedge\mkern-12.5mu\bigwedge}
               {\bigwedge\mkern-11mu\bigwedge}
    }
}

\begin{document}

\title{On the Borel completeness of torsion-free abelian groups and a Slaman-Wehner degree spectrum}
\author{George Crittenden and Matthew Harrison-Trainor\thanks{Harrison-Trainor was partially supported by the National Science Foundation under Grant DMS-2452105. This work began while Crittenden was an REU student at the University of Illinois Chicago funded under this grant.}}

\makeatletter
\def\blfootnote{\xdef\@thefnmark{}\@footnotetext}
\makeatother

\maketitle

\begin{abstract}
    It was recently shown by Paolini and Shelah, with another proof by Laskowski and Ulrich, that torsion-free abelian groups are Borel complete. This means that isomorphism for torsion-free abelian groups is as complicated as possible, so that torsion-free abelian groups do not admit any meaningful complete isomorphism invariants. We give a proof following the ideas of the Paolini-Shelah argument but with several expository advantages. In particular, the combinatorial part of the argument is much simplified. We make use of the ideas of this proof to show that there is a torsion-free abelian group with no computable copy but with a copy computable from every non-computable set. This answers a long-standing question of Downey and Goncharov.
\end{abstract}

\subsubsection*{Declaration of AI Use}

{\small The second author consulted ChatGPT 5.6 Sol in proving Lemma~\ref{lem:Zsigmondy}. The first author had an outline for an alternate, but more intricate, argument that avoided this lemma. AI tools were not otherwise used in the development of the initial draft. The second author used ChatGPT 5.6 Sol and 6 Astra to review the exposition and proofs as part of the editing process.}

\section{Introduction}

In their seminal paper, Friedman and Stanley \cite{FriedmanStanley} initiated the study of Borel reducibility for classes of structures. For countable languages $\mc{L}$, we work in the Polish space $\Mod(\mc{L})$ of countable $\mc{L}$-structures with domain $\mathbb{N}$. Given Borel classes $\mathbb{C}$ and $\mathbb{D}$, we say that \emph{$\mathbb{C}$ is Borel reducible to $\mathbb{D}$}, and write $\mathbb{C} \leq_B \mathbb{D}$, if there is a Borel map $\Phi$ such that, for $\mc{A},\mc{B} \in \mathbb{C}$,
\[ \mc{A} \cong \mc{B} \Longleftrightarrow \Phi(\mc{A}) \cong \Phi(\mc{B}).\]
If $\mathbb{C}$ is Borel reducible to $\mathbb{D}$, this means that the isomorphism relation on structures in $\mathbb{C}$ is at most as complicated as the isomorphism relation on structures in $\mathbb{D}$; in fact, isomorphism classes of structures in $\mathbb{D}$ form a set of complete invariants for $\mathbb{C}$.

Friedman and Stanley showed that certain classes of structures are \emph{Borel complete}, or \emph{on top for Borel reducibility}, in the sense that every other class of structures is Borel reducible to that class. Friedman and Stanley gave various examples of Borel complete classes including graphs, linear orders, fields, and trees. Later Camerlo and Gao \cite{CamerloGao} showed that Boolean algebras are Borel complete. Such classes of structures are unclassifiable. 

Friedman and Stanley asked whether torsion-free abelian groups are Borel complete. Hjorth \cite{Hjorth} showed that isomorphism is not Borel. Downey and Montalb\'an \cite{DowneyMontalban} showed that the set $\{ (G,H) : G \cong H\}$ of pairs of isomorphic torsion-free abelian groups is analytic-complete, which is a consequence of Borel completeness but not sufficient to prove Borel completeness. Finally, quite recently, Paolini and Shelah \cite{PaoliniShelah} and Laskowski and
Ulrich \cite{LaskowskiUlrich} both gave proofs, with the Paolini-Shelah proof appearing first but with some
errors which were later fixed and the Laskowski-Ulrich proof appearing in the interim.\footnote{Laskowski and Ulrich have since also written on the Borel complexity of modules in both the commutative \cite{LaskowskiUlrich2} and non-commutative \cite{LaskowskiUlrich3} cases. These include short proofs of the Borel completeness of torsion-free abelian groups.} The two proofs are quite different.

%Finally, in 2021, Paolini and Shelah posted a preprint (later published as \cite{PaoliniShelah}) with a proof of the Borel completeness of torsion-free abelian groups. Several rounds of corrections were posted between 2021 and 2023. In 2022 Laskowski and Ulrich posted the unpublished preprint \cite{LaskowskiUlrich} giving a proof along quite different lines. Laskowski and Ulrich later considered modules over commutative rings \cite{LaskowskiUlrich2} (including a succinct proof of the Borel completeness of torsion-free abelian groups) and recently  \cite{LaskowskiUlrich3} addressed the case of non-commutative rings.

\begin{theorem}[Paolini-Shelah \cite{PaoliniShelah} and Laskowski-Ulrich \cite{LaskowskiUlrich}]
    Torsion-free abelian groups are Borel complete.
\end{theorem}

\noindent One motivation of this paper is to digest the Paolini-Shelah proof of this fact. We give a proof along the same lines as the Paolini-Shelah proof, but making significant expository improvements. Like the Paolini-Shelah proof, our proof is separated into the construction of a combinatorial object and then a transformation of the combinatorial object into a group. One of the chief advantages of our argument is that the combinatorial object has fewer and simpler conditions and is more of a black box. The combinatorial object is described in Section~\ref{sec:combinatorial-structure} and constructed via a Fra\"{i}ss\'e-style construction. The details here are quite different and, we think, more explanatory. Like the Paolini-Shelah proof, we give a Borel reduction from the class of structures with two equivalence relations to the class of torsion-free abelian groups; the Borel completeness of the class of structures with two equivalence relations is quickly summarized in Section~\ref{sec:two-equivalence-relations}. In Section~\ref{sec:borel-completeness-tfag} we give the Borel reduction, using the combinatorial object to transform a structure with two equivalence relations into a torsion-free abelian group. This part of the argument uses essentially the same ideas as Paolini-Shelah, but adapted to our different combinatorial object.

The second motivation of this paper is that, following the resolution of such a prominent problem, one hopes that the new techniques introduced have further applications. We show that the ideas involved in the Borel completeness argument, and especially in the form they are developed in this paper, do lead to such applications. We do this by resolving an open question in computable structure theory. Though we use ideas from our presentation of the Paolini-Shelah argument, we would not be surprised if ideas from the Laskowski-Ulrich argument could equally be used; probably they each have their advantages. Before stating this question, we give a little background on the problem.

If a set $A$ is computable relative to all non-computable sets, then $A$ is computable \cite{KleenePost}. Thus there is no non-computable set computable from all non-computable sets. However, the non-computable sets do hold some non-trivial information in common: Answering a question of Lempp, Slaman \cite{Slaman98} and Wehner \cite{Wehner98} independently showed that there is a countable structure $\mc{A}$ which does not have a computable copy, but which does have a copy computable from each non-computable set. By \emph{computing a copy of $\mc{A}$} we mean computing an isomorphic structure with domain $\mathbb{N}$ and all of the functions, relations, and constants.

The \emph{degree spectrum} of a structure is the collection of Turing degrees of (the atomic diagrams of) its copies; these are among the central objects of study in computable structure theory (see, e.g., \cite{Knight86,Richter,CsimaKalimullin,highlow,lowercones,GMS,Nonarith}). Slaman and Wehner showed that the nonzero Turing degrees form the degree spectrum of a structure; we say that such a structure has \emph{Slaman-Wehner degree spectrum}. Downey asked whether such examples can be found in natural classes such as linear orders or abelian groups. This has been a long-standing open question. The still-unsolved case of linear orders is perhaps more well-known---see \cite{Miller01} for the best progress---but the case of abelian groups appears, for example, in the chapter on computable model theory by Fokina, Harizanov, and Melnikov \cite{FokinaHarizanovMelnikov}, Melnikov's survey on computable abelian groups \cite{Melnikov14} (where he also attributes the question to Goncharov), and Hirschfeldt's compilation of open questions in computability theory \cite{Hirschfeldt}. Melnikov \cite{Melnikov09} shows that there is a torsion-free abelian group whose degree spectrum consists of the non-low degrees and asks for not just an abelian group but a torsion-free abelian group with Slaman-Wehner spectrum.

In Section~\ref{sec:slaman-wehner} we answer this question by constructing a torsion-free abelian group with Slaman-Wehner degree spectrum.

\begin{theorem}\label{thm:slaman wehner}
    There is a torsion-free abelian group which has a copy computable from every non-computable set, but which has no computable copy.
\end{theorem}

\noindent This makes use of ideas from the argument for Borel completeness together with new ideas. We note that the second author had previously suggested (see \cite{GHTMT} and \cite{Faizrahmanov}) that to attack this problem one might consider the question of whether there is a Slaman-Wehner family which is closed under finite differences; this interesting question is still open and our resolution sidesteps this question.

We expect that there might be several other applications of the ideas developed in this paper in the construction of various torsion-free abelian groups. In Section~\ref{sec:open-problems} we summarize several open questions from the computable structure theory literature which we think would be promising to attack with these methods.

\section{The combinatorial structure}\label{sec:combinatorial-structure}

Let $\mc{L}$ be the language $\{\mathbb{X}\} \cup \{\mathbb{E}_n : n \in \mathbb{N}\}$. For each $n$, $\mathbb{E}_n$ is a relation on $2n$-tuples which we think of as an equivalence relation on $n$-tuples without repetition. We can think of the $\mathbb{E}_n$ as forming a single equivalence relation on the set of all finite tuples without repetition, which we denote as $\mathbb{E}$. We include $\mathbb{E}_0$, the equivalence relation on the empty tuple of length $0$, for convenience. $\mathbb{X}$ is a binary relation which we think of as an equivalence relation (on 1-tuples).   

\begin{definition}
	Let $\mc{A}$ be an $\mc{L}$-structure. A map $f \colon \mc{A} \to \mc{A}$ is a \emph{(strong) automorphism} if it is an automorphism and it preserves $\mathbb{E}$-equivalence classes. (It does not have to preserve $\mathbb{X}$-equivalence classes.)
\end{definition}

One can also think of a strong automorphism as an automorphism in the expanded language where each equivalence class is named by its own unary relation. A similar notion appears in the argument of Laskowski and Ulrich \cite{LaskowskiUlrich}.

\begin{definition}
	Let $\mc{A}$ be an $\mc{L}$-structure and let $f \colon \mc{A} \to \mc{A}$ be an automorphism of $\mc{A}$. Let $\alpha \colon \mc{A}/\mathbb{X}\to \mc{A}/\mathbb{X}$ be a permutation of the $\mathbb{X}$-equivalence classes. We say that \emph{$f$ is compatible with $\alpha$} (or \emph{acts as $\alpha$}) if $f(a) \in \alpha([a]_\mathbb{X})$ for each $a \in \mc{A}$.
\end{definition}

We can also extend both of these definitions to partial isomorphisms.

\begin{theorem}\label{thm:eq-structure}
	There is a countable $\mc{L}$-structure $\mc{A}$ such that:
	\begin{enumerate}
		\item There are infinitely many $\mathbb{X}$-equivalence classes, each of which is infinite.
		\item If $\alpha: \mc{A}/\mathbb{X} \to \mc{A}/\mathbb{X}$ is a permutation of the $\mathbb{X}$-equivalence classes, then there is a strong automorphism of $\mc{A}$ which acts as $\alpha$ on the $\mathbb{X}$-equivalence classes.
		%\item If $\alpha: M / X \to M / X$ is a permutation of the $\mathbb{X}$-equivalence classes, and $\bar{a} \mathbb{E}_n \bar{b}$ are $\mathbb{E}_n$-equivalent tuples such that for each $i$, $b_i \in \alpha([a_i])$, then there is an automorphism of $\mc{A}$ taking $\bar{a}$ to $\bar{b}$ and which acts as $\alpha$ on the $\mathbb{X}$-equivalence classes.
		\item Given distinct $\bar{x}_1,\ldots,\bar{x}_k \in \mc{A}^n$ in the same $\mathbb{E}_n$-equivalence class, with $k\geq 2$, there are distinct $(i_1,j_1)$ and $(i_2,j_2)$ with $x_{i_1}^{j_1},x_{i_2}^{j_2}$ such that
		\[ x_{i_1}^{j_1} \notin \{x_i^j : \text{$(i_1,j_1) \neq (i,j)$}\}\]
		and
		\[ x_{i_2}^{j_2} \notin \{x_i^j : \text{$(i_2,j_2) \neq (i,j)$}\}\]
	\end{enumerate}
\end{theorem}

To prove Theorem \ref{thm:eq-structure}, we will build $\mc{A}$ as the union of finite structures. Think of conditions (1) and (2) as positive requirements which we must meet along the way, and condition (3) as a negative requirement which we must maintain. To meet condition (1), we just need to add new equivalence classes (and new elements to those equivalence classes) at various stages through the construction. It is easy to do this while maintaining (3) by making all of the new tuples fall into new $\mathbb{E}$-equivalence classes.

                                                                                                                                                                                                                                                                                                                    	To meet (2), we act towards making $\mc{A}$ homogeneous in the following way. We will try to make the $\mathbb{E}$-equivalence classes the strong automorphism orbits of $\mc{A}$ using a back-and-forth argument where we extend strong partial isomorphisms. Suppose that at some stage we have constructed a finite structure $\mc{A}_i$ containing tuples $\bar{a} \mathbb{E}_n \bar{b}$, and that $\bar{a}'$ is a listing of the remaining elements of $\mc{A}_i$ other than $\bar{a}$. Then we will construct $\mc{A}_{i+1}$ containing a tuple $\bar{b}'$ and such that $\bar{a}\bar{a}' \mathbb{E}_{n+m} \bar{b} \bar{b}'$. Thus the partial isomorphism $\bar{a}\mapsto \bar{b}$ will extend to a partial isomorphism $\bar{a} \bar{a}' \mapsto \bar{b} \bar{b}'$, and at the end of the construction, a back-and-forth argument will yield (2). (In addition, the elements of $\bar{b}'$ should be in particular $\mathbb{X}$-equivalence classes, but this is not the main source of difficulty so we ignore it for now.) However, we must be careful that we maintain property (3) throughout the back-and-forth argument.

It is not the case that, given any finite structure $\mc{A}_i$ satisfying (3), we can find a $\bar{b}'$ as above while maintaining property (3). Consider the following example, where $e$ and $f$ are two $\mathbb{E}_2$-equivalence classes.
\[\begin{tikzpicture}[node distance=3cm]

\draw [label=f,fill=lightgray,color=lightgray] (-6,1.5) ellipse (1.3 and 0.4);
\draw [label=f,fill=lightgray,color=lightgray] (-2.5,1.5) ellipse (1.3 and 0.4);

\node [circle,fill=black,label=$a_1$] (v1) at (-7,1.5) {};
\node [circle,fill=black] (v7) at (-2,0) {};
\node [circle,fill=black] (v6) at (-3,0) {};

\node [circle,fill=black,label=$b_2$] (v5) at (-1.5,1.5) {};
\node [circle,fill=black,label=$b_1$] (v4) at (-3.5,1.5) {};
\node [circle,fill=black] (v2) at (-6,0) {};
\node [circle,fill=black,label=$a_2$] (v3) at (-5,1.5) {};

\draw  (v1) edge node[left,below, sloped] {$f$}  (v2);
\draw  (v3) edge node[left,below, sloped] {$f$}  (v2);
\draw  (v1) edge node[left,above, sloped] {$e$}  (v3);

\draw  (v4) edge node[left,above, sloped] {$e$} (v5);
\draw  (v4) edge node[left,below, sloped] {$f$}  (v6);
\draw  (v6) edge node[left,below, sloped] {$f$}  (v7);
\draw  (v7) edge node[left,below, sloped] {$f$}  (v5);

\end{tikzpicture}\]
If we have a map taking $\bar{a}= (a_1,a_2)$ to $\bar{b}=(b_1,b_2)$, then we would have to add a node $b'$ as in the following diagram:
\[\begin{tikzpicture}[node distance=3cm]

\draw [label=f,fill=lightgray,color=lightgray] (-6,1.5) ellipse (1.3 and 0.4);
\draw [label=f,fill=lightgray,color=lightgray] (-2.5,1.5) ellipse (1.3 and 0.4);

\node [circle,fill=black,label=$a_1$] (v1) at (-7,1.5) {};
\node [circle,fill=black] (v7) at (-2,0) {};
\node [circle,fill=black] (v6) at (-3,0) {};

\node [circle,fill=black,label=$b_2$] (v5) at (-1.5,1.5) {};
\node [circle,fill=black,label=$b_1$] (v4) at (-3.5,1.5) {};
\node [circle,fill=black] (v2) at (-6,0) {};
\node [circle,fill=black,label=$a_2$] (v3) at (-5,1.5) {};

\draw  (v1) edge node[left,below, sloped] {$f$}  (v2);
\draw  (v3) edge node[left,below, sloped] {$f$}  (v2);
\draw  (v1) edge node[left,above, sloped] {$e$}  (v3);

\draw  (v4) edge node[left,above, sloped] {$e$} (v5);
\draw  (v4) edge node[left,below, sloped] {$f$}  (v6);
\draw  (v6) edge node[left,below, sloped] {$f$}  (v7);
\draw  (v7) edge node[left,below, sloped] {$f$}  (v5);

\node [circle,fill=black,label=$b'$] (v8) at (-2.5,3) {};
\draw  (v4) edge node[left,above, sloped] {$f$}  (v8);
\draw  (v5) edge node[left,above, sloped] {$f$}  (v8);

\end{tikzpicture}\]
But this would contradict (3).

The problem is that we have allowed two different tuples $\bar{a}$ and $\bar{b}$ to be extended in two different, but incompatible, ways. In particular, (a) for $\bar{a}$, there are two tuples $\bar{c}_1,\bar{c}_2$, in the same equivalence class $f$, and not all contained in $\bar{a}$, such that each $x \in \bar{c}_1,\bar{c}_2$ which appears in exactly one place in $\bar{c}_1,\bar{c}_2$ also appears in $\bar{a}$, and (b) for $\bar{b}$, there are three tuples $\bar{d}_1,\bar{d}_2,\bar{d}_3$, in the same equivalence class $f$, and not all contained in $\bar{b}$, such that each $x \in \bar{d}_1,\bar{d}_2,\bar{d}_3$ which appears in exactly one place in $\bar{d}_1,\bar{d}_2,\bar{d}_3$ also appears in $\bar{b}$. Think of (a) and (b) as \emph{potential problems}. Each individually is not a problem, but they only become a problem when they both occur.

The solution is that the tuples $\bar{a}$ will be \emph{guarded} by a larger set, which we denote $\Cl(\bar{a})$. (The map $\bar{a} \to \Cl(\bar{a})$ will, as the notation suggests, be a closure operator.) The tuple $\bar{b}$ will similarly be guarded by a larger set $\Cl(\bar{b})$, and because $\bar{a}$ and $\bar{b}$ are $\mathbb{E}$-equivalent, we will ensure that they are guarded in the same way; that is, we will ensure that there is a strong partial isomorphism between their closures, mapping $\bar{a}$ to $\bar{b}$. Moreover, any potential problem for $\bar{a}$ of the type described above will happen entirely within $\Cl(\bar{a})$, and similarly for $\bar{b}$. Thus the strong partial isomorphism between their closures will let us resolve any potential problems in a compatible way.

In the case of $\bar{a} = (a_1,a_2)$ as in the diagrams above, $\Cl(\bar{a})$ would be $\{a_1,a_2,a'\}$. Then the tuple $\bar{b}$ is guarded by $\Cl(\bar{b}) = \{b_1,b_2,b'\}$, and $(b_1,b_2,b')$ is in the same $\mathbb{E}_3$-equivalence class as $(a_1,a_2,a')$.
\[\begin{tikzpicture}[node distance=3cm]

\draw [label=f,fill=lightgray,color=lightgray] (-6,1.5) ellipse (1.3 and 0.4);
\draw [label=f,fill=lightgray,color=lightgray] (-2.5,1.5) ellipse (1.3 and 0.4);

\node [circle,fill=black,label=$a_1$] (v1) at (-7,1.5) {};
\node [circle,fill=black,label=below:$b'$] (v7) at (-2.5,0) {};

\node [circle,fill=black,label=$b_2$] (v5) at (-1.5,1.5) {};
\node [circle,fill=black,label=$b_1$] (v4) at (-3.5,1.5) {};
\node [circle,fill=black,label=below:$a'$] (v2) at (-6,0) {};
\node [circle,fill=black,label=$a_2$] (v3) at (-5,1.5) {};

\draw  (v1) edge node[left,below, sloped] {$f$}  (v2);
\draw  (v3) edge node[left,below, sloped] {$f$}  (v2);
\draw  (v1) edge node[left,above, sloped] {$e$}  (v3);

\draw  (v4) edge node[left,above, sloped] {$e$} (v5);
\draw  (v4) edge node[left,below, sloped] {$f$}  (v7);
\draw  (v7) edge node[left,below, sloped] {$f$}  (v5);

\end{tikzpicture}\]
Thus the map $\bar{a} \mapsto \bar{b}$ can safely be extended to a map $\Cl(\bar{a}) \mapsto \Cl(\bar{b})$. And since all the potential problems for $\bar{a}$ and $\bar{b}$ occur inside $\Cl(\bar{a})$ and $\Cl(\bar{b})$, the map can safely be extended on the rest of the domain as well.

We will not be able to make the $\mathbb{E}$-equivalence classes the automorphisms orbits, but we will ensure that if $\bar{a}$ and $\bar{b}$ are both closed (i.e., $\Cl(\bar{a})=\bar{a}$ and $\Cl(\bar{b})=\bar{b}$) and $\mathbb{E}$-equivalent, then there is a strong automorphism taking one to the other. In the language of Fraisse limits, the structure $\mc{A}$ we build will be \textit{weakly homogeneous} and its age will have the \textit{cofinal amalgamation property}. See the introduction to \cite{KKKP} for some discussion.

We make a formal definition of a guarded equivalence structure, writing down all of the desired properties.

\begin{definition}
	A \emph{guarded equivalence structure} is an $\mc{L}$-structure $\mc{A}$ equipped with a \emph{guard map} $\Cl \colon [\mc{A}]^{<\mathbb{N}} \to [\mc{A}]^{<\mathbb{N}}$ such that:
	\begin{enumerate}[label=(P{\arabic*})]
		\item Whenever $\bar{a} \mathbb{E}_n \bar{b}$, the map $a_i \mapsto b_i$ is a strong partial isomorphism. \label{pt:same-type}
			
		\item Given distinct $\bar{x}_1,\ldots,\bar{x}_k \in \mc{A}^n$ in the same $\mathbb{E}_n$-equivalence class, with $k \geq 2$, there are distinct $(i_1,j_1)$ and $(i_2,j_2)$ with $x_{i_1}^{j_1},x_{i_2}^{j_2}$ such that
		\[ x_{i_1}^{j_1} \notin \{x_i^j : \text{$(i_1,j_1) \neq (i,j)$}\}\]
		and
		\[ x_{i_2}^{j_2} \notin \{x_i^j : \text{$(i_2,j_2) \neq (i,j)$}\}\]
		\label{pt:two-free}
		
		\item $\Cl$ is a \emph{closure operator}, i.e., for finite $U,V \subseteq \mc{A}$,
		\begin{enumerate}
			\item $U \subseteq \Cl(U)$,
			\item if $U \subseteq V$, then $\Cl(U) \subseteq \Cl(V)$,
			\item $\Cl(\Cl(U)) = \Cl(U)$.
		\end{enumerate}
		If $\Cl(U) = U$, then we say that $U$ is \emph{closed}.
		
		\label{pt:closure}

	\item There is no tuple in $\Cl(\bar{a})$ in the same $\mathbb{E}$-equivalence class as $\bar{a}$ other than $\bar{a}$ itself.
	
	\label{pt:unique}
	
	\item Given $\bar{a} \mathbb{E}_n \bar{b}$, there is a strong partial isomorphism $\Cl(\bar{a}) \leftrightarrow \Cl(\bar{b})$ mapping $\bar{a}$ to $\bar{b}$.
	
	\label{pt:extends}
	
	\item Given a tuple $\bar{a}$ and a sequence of distinct $n$-tuples $\bar{x}_1,\ldots,\bar{x}_k$ in the same $\mathbb{E}_n$-equivalence class, each not entirely contained in $\Cl(\bar{a})$, there is $x_{i^*}^{j^*}$ such that
	\[ x_{i^*}^{j^*} \notin \Cl(\bar{a}) \cup \{x_i^j : \text{$(i^*,j^*) \neq (i,j)$}\}.\]
	%and
	%\[ x_{i_2}^{j_2} \notin \{a_i\} \cup \{x_i^j : \text{$(i_2,j_2) \neq (i,j)$}\}\]
	
	\label{pt:two-free-closure}
	
	%\item If $\bar{a} \in \Cl(e)$, there are no $\bar{b}_1,\ldots,\bar{b}_k$, all in the same equivalence class, and not all contained in $\bar{a}$, such that each $x \in \bar{b}_1,\ldots,\bar{b}_k$ appears in at least two different places among $\bar{a},\bar{b}_1,\ldots,\bar{b}_k$.
	\end{enumerate}
	For guarded equivalence structures we strengthen the notion of \emph{strong isomorphism} to require that closure is preserved: $x \in \Cl(U) \Longleftrightarrow f(x) \in \Cl(f(U))$.
\end{definition}

One consequence of Properties \ref{pt:unique} and  \ref{pt:extends} is that if two tuples are $\mathbb{E}$-equivalent, then their closures are also $\mathbb{E}$-equivalent (after ordering the closures in the right way).

As mentioned previously, the proof of Theorem \ref{thm:eq-structure} will be a back-and-forth construction making the structure homogeneous. It will be convenient to fix ahead of time names for the $\mathbb{X}$-equivalence classes, so that we can talk about permutations acting on $\mathbb{X}$-equivalence classes that we have not yet put elements in. Let $X$ be a countably infinite set which we think of as naming the $\mathbb{X}$-equivalence classes.

The following lemma encapsulates a single step of the construction.

\begin{lemma}\label{lem:extension}
	Let $\mc{A}$ be a finite guarded equivalence structure. Let $\bar{a},\bar{a}' \in \mc{A}$ be such that $\bar{a} \mathbb{E}_n \bar{a}'$. Let $\alpha: X \to X$ be a permutation of the $\mathbb{X}$-equivalence classes such there is a strong partial isomorphism $\Cl(\bar{a}) \to \Cl(\bar{a}')$ compatible with $\alpha$. Then there is a finite guarded equivalence structure $\mc{B} \supseteq \mc{A}$ such there is a strong partial isomorphism $\mc{A} \to \mc{B}$ mapping $\bar{a}$ to $\bar{a}'$ and which is compatible with $\alpha$.
\end{lemma}
\begin{proof}
	By replacing $\bar{a}$ and $\bar{a}'$ with $\Cl(\bar{a})$ and $\Cl(\bar{a}')$, we may assume that $\bar{a}$ and $\bar{a}'$ are closed. Property \ref{pt:extends} says that, after ordering them in the right way, $\Cl(\bar{a})$ and $\Cl(\bar{a}')$ still satisfy the hypotheses of the lemma. Property \ref{pt:unique} says that a strong partial isomorphism $\mc{A} \to \mc{B}$ mapping $\Cl(\bar{a})$ to $\Cl(\bar{a}')$ must map $\bar{a}$ to $\bar{a}'$ and so we can make this replacement without loss of generality.
	
	For each $x \in \mc{A} - \bar{a}$, let $x'$ be a new element. Let $\mc{B} = \mc{A} \cup \{x' : x \in \mc{A} - \bar{a}\}$. Thus, for each element $x \in \mc{A}$, we have an element $x'$; let $\mc{A}' = \{x' : x \in \mc{A}\} = (\mc{B} - \mc{A}) \cup \bar{a}'$. Thus $\mc{A} \cap \mc{A}' = \bar{a}'$. (If $x \in \bar{a}$, then $x' \in \bar{a}'$; otherwise, $x'$ is one of the new elements in $\mc{B} - \mc{A}$.) Extend the equivalence relations $\mathbb{E}_n$ to $\mc{B}$ as follows:
	\begin{itemize}
		\item If $\bar{c} \subseteq \mc{A}$, then the $\mathbb{E}_n$-equivalence class of $\bar{c}$ is already defined.
		\item If $\bar{c} \subseteq \mc{A}'$, then writing $\bar{c} = (x_1',\ldots,x_n')$, put $(x_1',\ldots,x_n')$ in the same $\mathbb{E}_n$-equivalence class as $(x_1,\ldots,x_n)$.
		\item Otherwise, put $\bar{c}$ in a new $\mathbb{E}_n$-equivalence class.
	\end{itemize}
	Note that if we are in the first two cases at the same time, then $\bar{c} \subseteq \mc{A} \cap \mc{A}' = \bar{a}'$. Then, as $\bar{a} \mathbb{E} \bar{a}'$, the two cases are compatible by Property \ref{pt:same-type}. Also extend the equivalence relation $\mathbb{X}$ to $\mc{B}$ as follows:
	\begin{itemize}
		\item If $c \in \mc{A}$, then the $\mathbb{X}$-equivalence class of $c$ is already defined.
		\item If $c' \in \mc{A}'$, then put $c' \in \alpha([c]_\mathbb{X})$.
	\end{itemize}
	This is well-defined because for $c \in \bar{a}$, $c' \in \alpha([c]_\mathbb{X})$.
	
	We also need to extend $\Cl$ to all of $\mc{B}$. For each tuple $\bar{c}$, define $\Cl(\bar{c})$ as follows:
	\begin{itemize}
		\item If $\bar{c} \subseteq \mc{A}$, then $\Cl(\bar{c})$ is the same as it was in $\mc{A}$.
		\item If $\bar{c} \subseteq \mc{A}'$, then writing $\bar{c} = \bar{x}'$, $\Cl(\bar{c}) = \{ y' : y \in \Cl(\bar{x})\}$.
		\item Otherwise, $\Cl(\bar{c}) = \mc{B}$.
	\end{itemize}
	Again, the first two cases are compatible.
	
	The desired strong partial isomorphism $\mc{A} \to \mc{B}$ is just $x \mapsto x'$. This is a strong partial isomorphism which induces the map $\alpha$ on $\mathbb{X}$-equivalence classes; one can see this by the fact that the definitions above are well-defined.
	
	Finally, we must check that $\mc{B}$ is in fact a guarded equivalence structure by checking the six properties:
	\begin{enumerate}
		\item[\ref{pt:same-type}] Suppose that $\bar{x} \mathbb{E}_n \bar{y}$. Then we have several cases. If $\bar{x},\bar{y} \subseteq \mc{A}$, then this is because $\mc{A}$ satisfies Property \ref{pt:same-type}. If $\bar{x} \subseteq \mc{A}'$, then $\bar{x} = \bar{u}'$ for some $\bar{u} \subseteq \mc{A}$, and there is a strong partial isomorphism $\bar{u} \to \bar{x}$; so we may assume without loss of generality that $\bar{x} \subseteq \mc{A}$ (and similarly for $\bar{y}$). The final case is that $\bar{x} \nsubseteq \mc{A}$ and $\bar{x} \nsubseteq \mc{A}'$, in which case $\bar{x}$ is the only element of its $\mathbb{E}_n$-equivalence class; so $\bar{x} = \bar{y}$.
		
		\item[\ref{pt:two-free}] Suppose that we have distinct $\bar{b}_1,\ldots,\bar{b}_k$ in $\mc{B}$, all contained in the same $\mathbb{E}_n$-equivalence class. Since $\mc{A}$ satisfies Property \ref{pt:two-free}, if $\bar{b}_1,\ldots,\bar{b}_k$ are all contained in $\mc{A}$ then we are done. Similarly, we are done if they are all in $\mc{A}'$, as there are $\bar{c}_1,\ldots,\bar{c}_k$ in $\mc{A}$ such that $\bar{b}_1 = \bar{c}_1',\ldots,\bar{b}_k = \bar{c}_k'$, and we can apply \ref{pt:two-free} in $\mc{A}$ to $\bar{c}_1,\ldots,\bar{c}_k$. Third, if any $\bar{b}_i$ is not contained in either $\mc{A}$ or $\mc{A}'$, we are done as then it is the only tuple in its equivalence class.
		
		Otherwise, let $\bar{x}_1,\ldots,\bar{x}_\ell$ be those which are contained in $\mc{A}$ but not in $\mc{A}'$, let $\bar{y}_1,\ldots,\bar{y}_m$ be those that are contained in $\mc{A}'$ but not in $\mc{A}$, and let $\bar{z}_1,\ldots,\bar{z}_n$ be those that are contained in $\mc{A} \cap \mc{A}' = \bar{a}'$. There is at least one such tuple $\bar{x}$ and at least one such $\bar{y}$.
		
		Now $\bar{a}' = \Cl(\bar{a}')$ is closed and satisfies Property \ref{pt:two-free-closure}. Thus there is $x_{i_1}^{j_1}$ such that 
		\[ x_{i_1}^{j_1} \notin \bar{a}' \cup \{x_i^j : \text{$(i_1,j_1) \neq (i,j)$}\}.\]
		Since $x_{i_1}^{j_1} \notin \mc{A}'$,
		\[ x_{i_1}^{j_1} \notin \{x_i^j : \text{$(i_1,j_1) \neq (i,j)$}\} \cup \{y_i^j\} \cup \{z_i^j\}.\]
		Similarly, there is $y_{i_2}^{j_2}$ such that 
		\[ y_{i_2}^{j_2} \notin \bar{a}' \cup \{y_i^j : \text{$(i_2,j_2) \neq (i,j)$}\}.\]
		Since $y_{i_2}^{j_2} \notin \mc{A}$,
		\[ y_{i_2}^{j_2} \notin  \{x_i^j\} \cup \{y_i^j : \text{$(i_2,j_2) \neq (i,j)$}\} \cup \{z_i^j\}.\]
		This completes the verification of Property \ref{pt:two-free}.
		
		\item[\ref{pt:closure}] 
		\begin{enumerate}
			\item Fix $U \subseteq \mc{B}$. If $U \subseteq \mc{A}$, then $\Cl(U)$ was already defined in $\mc{A}$ so that $U \subseteq \Cl(U)$. If $U \subseteq \mc{A}'$, then $U = \{v' : v \in V\}$ for some $V \subseteq \mc{A}$, and $V \subseteq \Cl(V)$ in $\mc{A}$; thus $U \subseteq \Cl(U) = \{v' : v \in \Cl(V)\}$. Finally, if $U \nsubseteq \mc{A}$ and $U \nsubseteq \mc{A}'$, we defined $\Cl(U) = \mc{B} \supseteq U$.
			
			\item Fix $U \subseteq V \subseteq \mc{B}$. If $V \subseteq \mc{A}$, then $U \subseteq \mc{A}$, and $\Cl(U) \subseteq \Cl(V)$ was already true in $\mc{A}$. If $V \subseteq \mc{A}'$, then $U \subseteq \mc{A}'$, and there are sets $U^*,V^* \subseteq \mc{A}$ such that $U = \{u' : u \in U^*\}$ and $V = \{v' : v \in V^*\}$. We have $\Cl(U^*) \subseteq \Cl(V^*)$ in $\mc{A}$, and so $\Cl(U) = \{u' : u \in \Cl(U^*)\} \subseteq \{v' : v \in \Cl(V^*)\} = \Cl(V)$. Finally, if $V \nsubseteq \mc{A}$ and $V \nsubseteq \mc{A}'$, then $\Cl(U) \subseteq \mc{B} = \Cl(V)$.
			
			\item Fix $U \subseteq \mc{B}$.  If $U \subseteq \mc{A}$, then $\Cl(U)$ was already defined in $\mc{A}$ so that $\Cl(U) = \Cl(\Cl(U))$. If $U \subseteq \mc{A}'$, then $U = \{v' : v \in V\}$ for some $V \subseteq \mc{A}$, and $\Cl(V) = \Cl(\Cl(V))$ in $\mc{A}$; thus \[\Cl(U) = \{v' : v \in \Cl(V)\} = \{v' : v \in \Cl(\Cl(V))\} = \Cl(\Cl(U)).\] Finally, if $U \nsubseteq \mc{A}$ and $U \nsubseteq \mc{A}'$, we defined $\Cl(U) = \mc{B} = \Cl(\mc{B})$.
		\end{enumerate}
		\item[\ref{pt:unique}] We have three cases. If $\bar{x} \subseteq \mc{A}$, then this is because Property \ref{pt:unique} holds in $\mc{A}$. If $\bar{x} \subseteq \mc{A}'$, then $\bar{x} = \bar{u}'$ for some $\bar{u} \subseteq \mc{A}$, and $\bar{u}$ is the only tuple in its $\mathbb{E}$-equivalence class in $\Cl(\bar{u})$. Then there is a strong partial isomorphism $\Cl(\bar{u}) \to \Cl(\bar{u}') = \Cl(\bar{x})$, and so $\bar{x}$ is the only tuple in its $\mathbb{E}$-equivalence class in $\Cl(\bar{x})$. Finally, if $\bar{x} \nsubseteq \mc{A}$ and $\bar{x} \nsubseteq \mc{A}'$, then $\bar{x}$ is the only tuple in its $\mathbb{E}$-equivalence class in $\mc{B} = \Cl(\bar{x})$.
		\item[\ref{pt:extends}] Suppose that $\bar{x} \mathbb{E}_n \bar{y}$. We have several cases. If $\bar{x}, \bar{y} \subseteq \mc{A}$, then this follows from Property \ref{pt:extends} in $\mc{A}$. If $\bar{x} \subseteq \mc{A}'$, then $\bar{x} = \bar{u}'$ for some $\bar{u} \subseteq \mc{A}$. There is a strong partial isomorphism $\Cl(\bar{u}) \to \Cl(\bar{u}') = \Cl(\bar{x})$, and so without loss of generality we may assume that $\bar{x} \subseteq \mc{A}$ (and similarly for $\bar{y}$). The last case is that $\bar{x}$ or $\bar{y}$ are not in $\mc{A}$ or $\mc{A}'$, in which case they are the only tuples of their equivalence class in $\Cl(\bar{x}) = \Cl(\bar{y}) = \mc{B}$, and so $\bar{x} = \bar{y}$.
		
		\item[\ref{pt:two-free-closure}] Suppose that $\bar{b}_1,\ldots,\bar{b}_k$ are distinct tuples in the same $\mathbb{E}_n$-equivalence class, and each of them is not entirely contained in $\Cl(\bar{c})$.
		
		First note that $\bar{c} \subseteq \mc{A}$ or $\bar{c} \subseteq \mc{A}'$. Otherwise, we would have that $\Cl(\bar{c}) = \mc{B}$ and each of $\bar{b}_1,\ldots,\bar{b}_k$  would be entirely contained in $\Cl(\bar{c})$. We will assume for now that $\bar{c} \subseteq \mc{A}$; the case $\bar{c} \subseteq \mc{A}'$ is similar.
		
		If all of $\bar{b}_1,\ldots,\bar{b}_k$ are contained entirely in $\mc{A}$, then we can just apply Property \ref{pt:two-free-closure} in $\mc{A}$. So we may assume that at least one of them is not entirely contained in $\mc{A}$. Without loss of generality, say that these are $\bar{b}_1,\ldots,\bar{b}_\ell$. Then these must be entirely contained in $\mc{A}'$, as otherwise they would be the only element of their $\mathbb{E}$-equivalence classes, so we would have $k = \ell = 1$ and we could pick any $b_1^j \notin \mc{A}$.
		
		Now there are $\bar{d}_1,\ldots,\bar{d}_\ell$ in $\mc{A}$ such that $\bar{b}_i = \bar{d}_i'$. Note that since none of the $\bar{b}_i$ are entirely contained in $\bar{a}' = \Cl(\bar{a}') = \mc{A} \cap \mc{A}'$, none of the $\bar{d}_i$ are entirely contained in $\bar{a} = \Cl(\bar{a})$. So by Property \ref{pt:two-free-closure} in $\mc{A}$, there are $i^*$ and $j^*$ such that
		\[ d_{i^*}^{j^*} \notin \Cl(\bar{a}) \cup \{d_i^j : (i,j) \neq (i^*,j^*)\}.\]
		Then
		\[ b_{i^*}^{j^*} = (d_{i^*}^{j^*})' \notin \Cl(\bar{a}') \cup \{d_i^{j\prime} : (i,j) \neq (i^*,j^*), \; i \leq \ell\} = (\mc{A} \cap \mc{A}') \cup \{b_i^j : (i,j) \neq (i^*,j^*), \; i \leq \ell\}.\]
		Since $\bar{c} \subseteq \mc{A}$, and $\bar{b}_{\ell+1},\ldots,\bar{b}_{k} \subseteq \mc{A}$, and $b_{i^*}^{j^*} \in \mc{A}'$,
		\[ b_{i^*}^{j^*} \notin \Cl(\bar{c}) \cup \{b_i^j : (i,j) \neq (i^*,j^*)\}.\]
		This verifies Property \ref{pt:two-free-closure}.
	\end{enumerate}
	Thus $\mc{B}$ is a guarded equivalence structure.
\end{proof}

The following remark about the lemma will save us some time in the full construction by giving us for free that each equivalence class is infinite.

\begin{remark}\label{rem:new-elements}
	In Lemma \ref{lem:extension}, if $\bar{a} = \varnothing$ and $\Cl(\bar{a}) = \varnothing$, then the partial isomorphism $\mc{A} \to \mc{B}$ maps $\mc{A}$ into $\mc{B} - \mc{A}$. (That is, the image of each element of $\mc{A}$ is a new element.)
\end{remark}

We are now ready to prove Theorem \ref{thm:eq-structure}. The construction is essentially to iterate the previous lemma over and over again for different choices of tuples and maps.

\begin{proof}[Proof of Theorem \ref{thm:eq-structure}]
	We construct $\mc{A}$ as a union $\bigcup \mc{A}_i$ of finite guarded equivalence structures $\mc{A}_i$. We can choose ahead of time a countable set to be the domain of $\mc{A}$. Begin with $\mc{A}_0$ the structure with a single element $a$, with $\Cl(\varnothing) = \varnothing$ and $\Cl(a) = \{a\}$.
	
	\medskip
	
	Let $(\bar{a}_i)_{i \in \mathbb{N}}$ be a listing of all finite tuples without repetition from $\mc{A}$, including the empty tuple. Let $(\alpha_i)_{i \in \mathbb{N}}$ be a listing of the finite partial permutations $X \to X$ of the $\mathbb{X}$-equivalence classes. At each stage $s+1 = \la i,j,k,\ell \ra + 1$ of the construction, we will ensure that if $\bar{a}_i$ and $\bar{a}_j$ are in $\mc{A}_s$, $\bar{a}_i \mathbb{E}_n \bar{a}_j$, and there is a strong partial isomorphism $\Cl(\bar{a}_i) \to \Cl(\bar{a}_j)$ compatible with $\alpha_k$, then there is a strong partial isomorphism $\mc{A}_s \to \mc{A}_{s+1}$ mapping $\bar{a}_i$ to $\bar{a}_j$ and which is compatible with $\alpha_k$ on the $\mathbb{X}$-equivalence classes.

	\medskip

	\noindent \textit{Stage $s+1 = \la i,j,k,\ell \ra + 1$.} If either $\bar{a}_i$ or $\bar{a}_j$ are not in $\mc{A}_s$, or if it is not the case that $\bar{a}_i \mathbb{E}_n \bar{a}_j$, or if it is not the case that there is a strong partial isomorphism $\Cl(\bar{a}_i) \to \Cl(\bar{a}_j)$ compatible with $\alpha_k$, then let $\mc{A}_{s+1} = \mc{A}_s$. In this case we say that the construction was \emph{passive} at stage $s+1$. Otherwise let $\mc{A}_{s+1}$ be the structure obtained by applying Lemma \ref{lem:extension} to $\bar{a} = \bar{a}_i$, $\bar{a}' = \bar{a}_j$, and $\alpha$ any permutation of $X$ extending $\alpha_k$. In this case, we say that the construction was \emph{active} at stage $s+1$.
	
	\medskip
	
	Now let us check that the construction works. 
	\begin{enumerate}
		\item Let $a$ be the single element of $\mc{A}_0$. Let $x_1 = a / \mathbb{X}$ be the equivalence class of $a$, and let $x_2 \in X$ be an equivalence class, possibly with $x_2 = x_1$. Let $i$ and $j$ be such that $\bar{a}_i = \varnothing$ and $\bar{a}_j = \varnothing$. Let $k$ be such that $\alpha_k$ maps $x_1 \mapsto x_2$. Then at each stage $s+1=\la i,j,k,\ell \ra+1$, there is an element $a'$ of $\mc{A}_{s+1}$ such that $[a']_{\mathbb{X}} = \alpha_k([a]_{\mathbb{X}})= \alpha_k(x_1) = x_2$. By Remark \ref{rem:new-elements}, $a' \notin \mc{A}_s$ and so at this stage we introduced a new element in the equivalence class $x_2$. Since this happens at infinitely many stages for each $x_2$, every equivalence class is infinite. Thus there are infinitely many $\mathbb{X}$-equivalence classes, each of which is infinite.
		
		\item Let $\alpha: X \to X$ be a permutation of the $\mathbb{X}$-equivalence classes. We want to show that there is an automorphism of $\mc{A}$ which acts as $\alpha$ on the $\mathbb{X}$-equivalence classes. We build this automorphism via a back-and-forth method.
		
		We will build, stage-by-stage, tuples $\bar{b}_t$ and $\bar{c}_t$ and a strong partial isomorphism $\gamma_t$ mapping $\bar{b}_t \mapsto \bar{c}_t$ compatible with $\alpha$ on the $\mathbb{X}$-equivalence classes. The tuples $\bar{b}_t$ and $\bar{c}_t$ will always be closed and contained in $\mc{A}_{s_t}$ for some $s_t$ that we will define at this stage.
		
		Begin with $\bar{b}_0 = \varnothing$, $\bar{c}_0 = \varnothing$, and $\gamma_0 = \varnothing$. We will have $s_0 = 0$ so that $\mc{A}_{s_0} = \mc{A}_0$.
		
		\begin{description}
			\item[\textit{Stage $t+1 = 2u+1$.}] Let $i$ and $j$ be such that $\bar{a}_i = \bar{b}_t$ and $\bar{a}_j = \bar{c}_t$. Let $k$ be such that $\alpha_{k}$ is the restriction of $\alpha$ to the $\mathbb{X}$-equivalence classes appearing in $\mc{A}_{s_t}$. Let $\ell$ be sufficiently large that $\la i,j,k,\ell\ra \geq s_{t}$, and let $s_{t+1} = \la i,j,k,\ell\ra + 1$.
			
			Then, as $\bar{a}_{i} = \bar{b}_t$ and $\bar{a}_{j}= \bar{c}_t$ are in $\mc{A}_{s_t} \subseteq \mc{A}_{s_{t+1}-1}$, and $\gamma_t$ is a strong partial isomorphism mapping $\bar{a}_{i} = \bar{b}_t$ to $\bar{a}_{j} = \bar{c}_t$ and compatible with $\alpha_{k}$, the construction was \textit{active} at stage $s_{t+1}$. 
			
			So there is a map $\gamma \colon \mc{A}_{s_{t+1}-1} \to \mc{A}_{s_{t+1}}$ mapping $\bar{a}_{i} = \bar{b}_t \mapsto \bar{a}_{j} = \bar{c}_t$ and which is compatible with $\alpha_{k}$ on the $\mathbb{X}$-equivalence classes. Extend $\bar{b}_t$ to a tuple $\bar{b}_{t+1}$ listing the elements of $\mc{A}_{s_t}$, and let $\bar{c}_{t+1} = \gamma(\bar{b}_{t+1})$. Let $\gamma_{t+1}$ be $\gamma$ restricted to $\bar{b}_{t+1}$.
			
			\item[\textit{Stage $t+1 = 2u+2$.}] This is essentially the same as the previous stage, but back instead of forth. Let $i$ and $j$ be such that $\bar{a}_i = \bar{c}_t$ and $\bar{a}_j = \bar{b}_t$. Let $k$ be such that $\alpha_{k}$ is the restriction of $\alpha^{-1}$ to the $\mathbb{X}$-equivalence classes appearing in $\mc{A}_{s_t}$. Let $\ell$ be sufficiently large that $\la i,j,k,\ell\ra \geq s_{t}$, and let $s_{t+1} = \la i,j,k,\ell\ra + 1$.
			
			Then, as $\bar{a}_{i} = \bar{c}_t$ and $\bar{a}_{j}= \bar{b}_t$ are in $\mc{A}_{s_t} \subseteq \mc{A}_{s_{t+1}-1}$, and $\gamma^{-1}_t$ is a strong partial isomorphism mapping $\bar{a}_{i} = \bar{c}_t$ to $\bar{a}_{j} = \bar{b}_t$ and compatible with $\alpha_{k}$, the construction was \textit{active} at stage $s_{t+1}$. 
			
			So there is a map $\gamma^* \colon \mc{A}_{s_{t+1}-1} \to \mc{A}_{s_{t+1}}$ mapping $\bar{a}_{i} = \bar{c}_t \mapsto \bar{a}_{j} = \bar{b}_t$ and which is compatible with $\alpha_{k}$ on the $\mathbb{X}$-equivalence classes. Extend $\bar{c}_t$ to a tuple $\bar{c}_{t+1}$ listing the elements of $\mc{A}_{s_t}$, and let $\bar{b}_{t+1} = \gamma^*(\bar{c}_{t+1})$. Let $\gamma_{t+1}$ be ${\gamma^*}^{-1}$ restricted to $\bar{b}_{t+1}$.
		\end{description}
		
		Continuing in this way, we define a sequence of $\mc{A}_{s_t}$ and a nested sequence of strong partial isomorphisms $\gamma_{t}$ such that each $\gamma_t$ is compatible with $\alpha$ on the $\mathbb{X}$-equivalence classes. Let $\gamma = \bigcup \gamma_i$. For each odd $t \geq 1$, the domain of $\gamma_t$ contains $\mc{A}_{s_{t-1}}$, and for each even $t$, its range contains $\mc{A}_{s_{t-1}}$. Thus $\gamma$ is an automorphism of $\mc{A}$.
		
		\item Because each $\mc{A}_s$ is a guarded equivalence structure, they each satisfy Property \ref{pt:two-free}. Because this is a universal property, it is inherited by the limit $\mc{A}$.
	\end{enumerate}

	This completes the proof of the theorem.
\end{proof}

\section{Borel completeness of two equivalence relations}\label{sec:two-equivalence-relations}

\begin{definition}
    For $\ell \in \mathbb{N}$, let $\Eq_\ell$ denote the set of models in the language $(E_1,\cdots,E_{\ell})$, where each $E_i$ for $1\leq i\leq \ell$ is an equivalence relation.
\end{definition}

Like Paolini and Shelah, we will use $\Eq_2$ for the proof of Borel completeness. In Section~\ref{sec:slaman-wehner}, where we show that there is a torsion-free abelian group with Slaman-Wehner degree spectrum, we will use $\ell=0$. It is for this reason that we leave $\ell$ as a variable for now.

As shown in \cite{FriedmanStanley}, $\Eq_2$ is Borel complete, and thus to show that torsion-free abelian groups are Borel complete, it suffices to give a Borel reduction from $\Eq_2$. We will briefly outline the proof of the Borel completeness of $\Eq_2$, as it is potentially useful for the reader to understand the full Borel reduction from graphs to groups. The construction maps vertices to $E_1$-equivalence classes and represents edges using $E_2$. To show this, we construct, for each graph $G=(V,E)$, a corresponding structure in $\Eq_2$. The domain will be
\[
\mc{N}=\{(x,y)\in V^2: x\neq y\}.
\]
Set two elements to be $E_1$-equivalent if and only if they agree on their first coordinate; that is, put $(x,y)E_1(x,z)$. Writing $x\mathrel{E}y$ when there is an edge between $x$ and $y$ in $G$, set two elements to be $E_2$-equivalent if and only if they are equal or they are of the form $(x,y)$ and $(y,x)$ with $x\mathrel{E}y$. Thus,
\[
E_1=\{((x,y),(x,z))\in \mc{N}^2\}
\]
and
\[
E_2=\{((x,y),(x,y))\in \mc{N}^2\}\cup\{((x,y),(y,x))\in \mc{N}^2: x\mathrel{E}y\}.
\]
Then, we can represent a graph as follows:

\begin{align*}
    \begin{tikzpicture}[scale=0.75, transform shape]
    %graph structure
    \node [scale=0.5,circle,fill=black,label=below:$c$]  (v1) at (-8,1.5) {};
    \node [scale=0.5,circle,fill=black,label=below:$b$]  (v2) at (-11,1.5) {};
    \node [scale=0.5,circle,fill=black,label=below:$a$]  (v3) at (-14,1.5) {};
    \node [scale=0.5,circle,fill=black, label=$d$]  (v4) at (-11,4.5) {};
    \draw  (v1) edge node[left,below, sloped] {}  (v2);
    \draw  (v4) edge node[left,below, sloped] {}  (v2);
    \draw  (v1) edge node[left,below, sloped] {}  (v4);
    \draw  (v3) edge node[left,below, sloped] {}  (v2);
    \node[font=\LARGE] at (-6.6,3) {$\Longrightarrow$};
    %equivalence structure
    \def\eTwoBand#1#2#3#4{%
        \begin{scope}[shift={(#1,#2)}, rotate=#3, xscale=#4]
        \filldraw [fill=gray!40!white, draw=black]
            (0,0.34)
            .. controls (0.556,0.34) and (0.864,0.3) .. (0.969,0.08)
            .. controls (1,0.03) and (1,-0.03) .. (0.969,-0.08)
            .. controls (0.864,-0.3) and (0.556,-0.34) .. (0,-0.34)
            .. controls (-0.556,-0.34) and (-0.864,-0.3) .. (-0.969,-0.08)
            .. controls (-1,-0.03) and (-1,0.03) .. (-0.969,0.08)
            .. controls (-0.864,0.3) and (-0.556,0.34) .. (0,0.34)
            -- cycle;
        \end{scope}%
    }
    \eTwoBand{-2}{1.5}{0}{1.62}
    \eTwoBand{2}{1.5}{0}{1.62}
    \eTwoBand{0}{3.15}{90}{1.62}
    \eTwoBand{2.1457}{3.4043}{135}{2.29103}
    \draw (-4,1.5) circle (1.6) node[below=1.75] {$(a,x)/\rE_1$};
    \draw (0,1.5) circle (1.6) node[below=1.75] {$(b,x)/\rE_1$};
    \draw (4,1.5) circle (1.6) node[below=1.75] {$(c,x)/\rE_1$};
    \draw (0,4.8) circle (1.6) node[above=1.75] {$(d,x)/\rE_1$};
    \tikzset{element label/.style={scale=0.75, fill=white, fill opacity=0.82, text opacity=1, inner sep=0.8pt, rounded corners=1.5pt}}
    \fill (-4,2.05) circle (1.5pt);
    \node [element label, anchor=south, yshift=3pt] at (-4,2.05) {$(a,d)$};
    \fill (-4.95,1.5) circle (1.5pt);
    \node [element label, anchor=north, yshift=-3pt] at (-4.95,1.5) {$(a,c)$};
    \fill (-3.05,1.5) circle (1.5pt);
    \node [element label, anchor=north, yshift=-3pt] at (-3.05,1.5) {$(a,b)$};
    \fill (-0.95,1.5) circle (1.5pt);
    \node [element label, anchor=north, yshift=-3pt] at (-0.95,1.5) {$(b,a)$};
    \fill (0.95,1.5) circle (1.5pt);
    \node [element label, anchor=north, yshift=-3pt] at (0.95,1.5) {$(b,c)$};
    \fill (3.05,1.5) circle (1.5pt);
    \node [element label, anchor=north, yshift=-3pt] at (3.05,1.5) {$(c,b)$};
    \fill (4.95,1.5) circle (1.5pt);
    \node [element label, anchor=north, yshift=-3pt] at (4.95,1.5) {$(c,a)$};
    \fill (0,4.05) circle (1.5pt);
    \node [element label, anchor=east, xshift=-3pt] at (0,4.05) {$(d,b)$};
    \fill (0,5.65) circle (1.5pt);
    \node [element label, anchor=south, yshift=3pt] at (0,5.65) {$(d,a)$};
    \fill (0,2.25) circle (1.5pt);
    \node [element label, anchor=west, xshift=3pt] at (0,2.25) {$(b,d)$};
    \fill (3.3257,2.2243) circle (1.5pt);
    \node [element label, anchor=south west, xshift=2pt, yshift=2pt] at (3.3257,2.2243) {$(c,d)$};
    \fill (0.9657,4.5843) circle (1.5pt);
    \node [element label, anchor=south east, xshift=-2pt, yshift=2pt] at (0.9657,4.5843) {$(d,c)$};
    \end{tikzpicture}
\end{align*}

\section{The Borel completeness of torsion-free abelian groups}\label{sec:borel-completeness-tfag}
For the remainder of this section, fix $\ell \in \mathbb N$ and fix a structure $\mathcal{A}$ as in Theorem \ref{thm:eq-structure}. Given $\mc{N} \in \Eq_\ell$, we will produce a torsion-free abelian group. Fix a bijection $g: \mc{A}/\mathbb{X} \to \mc{N}$. We form the free abelian group $\bigoplus_{x \in \mc{A}} \mathbb{Z} x$ generated over elements of $\mc{A}$ and its divisible closure $\bigoplus_{x \in \mc{A}} \mathbb{Q} x$. Our group $G^{\mc{N},g}$ will be intermediate between these, $\bigoplus_{x \in \mc{A}} \mathbb{Z} x \subseteq G^{\mc{N},g} \subseteq \bigoplus_{x \in \mc{A}} \mathbb{Q} x$. An element of $\mc{N}$ will be represented by the elements of its image under $g$, which is an $\mathbb{X}$-equivalence class of elements of $\mc{A}$. We must then show that we can recover the isomorphism type of $\mc{N}$ from $G^{\mc{N},g}$, i.e., $G^{\mc{N},g} \cong G^{\mc{M},h}$ if and only if $\mc{N} \cong \mc{M}$.

\subsection{Construction of a group}

\begin{definition} \label{def:atomic-types}
	Let $\n\in \Eq_\ell$, and let $S$ be the set of all atomic types. Given a bijection $g: \mc{A}/\mathbb{X}\rightarrow\n$, define the following functions:
	\begin{enumerate}
		\item Let $g: \mc{A}\twoheadrightarrow \n$ be the induced map such that for $x\in \mc{A}$, $g(x)=g(x/\mathbb{X})\in \n$.
		
		\item Let $\tp_{g, \n}=\tp_{g}:\mc{A}^{<\mathbb{N}}\rightarrow S$ be such that for $\bar{x}\in \mc{A}^{<\mathbb{N}}$, $\tp_g(\bar{x})=\tp_{g}(g(\bar{x}))\in S$ is the atomic type of $g(\bar{x})\in \n$.
		
	\end{enumerate}
	Note that we have made some abuse of notation.
	
\end{definition}

\[
\xymatrix@C=4em{
    \mc{A} \ar@{->>}[r] \ar@/^1.5pc/@{.>}[rr]^g & \mc{A}/\mathbb{X} \ar[r]^{\sim}_g & \mc{N}
}
\]

\begin{notation}
	Let $\seq=\{\bar{x}\in \mc{A}^k: x_i\neq x_j \text{ for }i\neq j\}$ be the set of $k$-tuples without repetition from $\mc{A}$.
\end{notation}

% Given some $\n\in \Eq$ and a bijection $g:\n\rightarrow \mc{A}/\mathbb{X}$, the function $\tp_{g}: \mc{A}\rightarrow S$ induces an equivalence relation on $k$-tuples in $\mc{A}$, for every $k \in \mathbb{N}$. That is, $\bar{x},\bar{y}\in \seq$ can be considered equivalent if $\tp_{g}(\bar{x})=\tp_{g}(\bar{y})$. We can then take the intersection of the induced equivalence relation on $\mc{A}$ and the $\mathbb{E}_n$ to get an equivalence relation on $\mc{A}$ that reflects the structure of $\n$. 
% 
% 
% \begin{definition} Let $\n\in \Eq$ and a bijection $g:\mc{A}/\mathbb{X}\rightarrow \n$. For every $n \in \mathbb{N}$, the following is an equivalence relation on $\mc{A}$:
% 	\begin{align*}
% 		E^{\n,g}_{n}=\Es{g}{n}=\{(\bar{x},\bar{y})\in \mc{A}^n\times \mc{A}^n:(\bar{x},\bar{y})\in \mathbb{E}_n\text{ and }\tp_{g}(\bar{x})=\tp_{g}(\bar{y})\}
% 	\end{align*}
% 	When considering the equivalence classes for every $n \in \mathbb{N}$, we write $E^g=\bigcup_{n \in \mathbb{N}}\Es{g}{n}$.
% \end{definition}

\begin{definition} \label{def:p}   Let $(p_{\varepsilon,t,\bar{q}}: \varepsilon\in \seq/\mathbb{E}_k, t\in S, \bar{q}\in (\Znz)^k, k \in \mathbb{N})$ be a sequence of prime numbers listed without repetition, where if $q_i\in\bar{q}$, then $p_{\varepsilon, t,\bar{q}}\nmid q_i$.
	
\end{definition}

Given $\n\in \Eq$ and a bijection $g:\mc{A}/\mathbb{X}\rightarrow\n$, the function $\tp_{g}: \mc{A}^{<\mathbb{N}}\rightarrow S$ induces an equivalence relation on $k$-tuples in $\mc{A}$: $\bar{x},\bar{y}\in \seq$ are equivalent if $\tp_{g}(\bar{x})=\tp_{g}(\bar{y})$. We can then take the intersection of this equivalence relation on $\seq$ and $\mathbb{E}_k$ to get an equivalence relation on $\seq$ that reflects the structure of $\n$. When we define the $p_{\varepsilon,t,\bar{q}}$, we can think of the pair $(\varepsilon,t)$ as specifying the set $\{\bar{x}\in\varepsilon:\tp_g(\bar{x})=t\}$, which is an equivalence class of this new relation when nonempty. This is a helpful point of view to keep in mind as part of understanding the rest of the proof, but it ends up being too cumbersome to use formally. (The issue is that one cannot identify equivalence classes across different bijections $g$ and $h$.)

We are now ready to define the torsion-free abelian groups of the Borel reduction.

\begin{definition} \label{def:group-structure}
    Let $\n\in \Eq$ and let $g:\mc{A}/\mathbb{X}\rightarrow \n$ be a bijection. Define $G^{\n,g}=G^g$ to be the subgroup of $\bigoplus_{x\in \mc{A}}\mathbb{Q}x$ generated by:
    \begin{itemize}
        \item The subgroup $\bigoplus_{x\in \mc{A}}\mathbb{Z}x\subset\bigoplus_{x\in \mc{A}}\mathbb{Q}x$.
        \item The elements $\{p_{\varepsilon,t,\bar{q}}^{-m}\sum_{j<n}q_jy_j : m,n\in\mathbb{N}, \bar{y}\in\text{seq}_n(\mc{A}), \bar{q}\in(\Znz)^n, \varepsilon=\bar{y}/\mathbb{E}_n, t=\tp_g(\bar{y})\}$.
    \end{itemize}
   As $\bigoplus_{x\in \mc{A}}\mathbb{Q}x$ is a torsion-free abelian group,  the group $G^g$ is as well.
\end{definition}

This construction can easily be seen to yield a continuous (hence Borel) map taking presentations $\mc{N} \in \Eq_\ell$ with domain $\mathbb{N}$ to presentations of torsion-free abelian groups with domain $\mathbb{N}$ as follows. Fix ahead of time an indexing of $\mc{A} / \mathbb{X}$ and a list of primes as in Definition \ref{def:p}. Given $\mc{N}$, associate the $n$th element of $\mc{N}$ to the $n$th $\mathbb{X}$-equivalence class. This gives us our bijection $g : \mc{A} / \mathbb{X} \to \mc{N}$. Using this, we can determine $\tp_g(\bar{y})$ for each tuple $\bar{y}$ and construct a presentation of the group $G^g$ as in Definition~\ref{def:group-structure}. In the next subsection, we will show that this is the desired Borel reduction. First, we will prove the key lemma which allows us to understand the structure of $G^g$. 

\begin{notation} Let $G^g$ be as above, and let $p$ be a prime number.
    \begin{enumerate}
       \item Let $\Bp=\{a\in G^g : p^m|a \text{ for all }m \in \mathbb{N}\}\subset G^g$ denote the subgroup of $G^g$ that is infinitely $p$-divisible. Note this is a pure subgroup of $G^g$.

       \item  Let $\Qp=\{\frac{q_1}{q_2}\in\mathbb{Q}: q_1\in\mathbb{Z},q_2\in\Znz,\text{ and }p \nmid q_2\} $ be the localization of $\mathbb{Z}$ at the prime ideal $(p)$.

       \item Let $\nBp=\bigoplus_{x\in \mc{A}}\Qp x$ be the free $\Qp$-module over $\mc{A}$.

       \item For $a\in G^g$, with $a=\sum_{j<n}q_jy_j$, where $q_j\in \Qnz$ and $\bar{y}\in \text{seq}_n(\mc{A})$, let:
       \begin{enumerate}
           \item $\text{supp}(a)=\{y_j:j<n\}$ be the \emph{support} of $a$;
           \item $\text{supp}_p(a)=\{y_j:j<n, q_j\notin\Qp\}$ be the \emph{$p$-support} of $a$.
       \end{enumerate}

    \end{enumerate}
\end{notation}

%a-modulo-p in the group%
The following lemma is the key to the whole argument. A difficulty in building torsion-free abelian groups is that unexpected cancellation in sums of generators can introduce elements that we intended to exclude. This lemma says that in certain non-trivial sums there are always some terms that only show up a single time and so do not cancel. This comes from property (3) of the combinatorial structure $\mc{A}$ obtained from Theorem \ref{thm:eq-structure}, and means that we are able to keep certain elements out of the group and so have more control over what we have constructed.

\begin{lemma}
\label{lem:a-modulo-p}
    Let $a\in G^g$, and let $p=p_{\varepsilon,t,\bar{q}}$ where $\varepsilon\in\mathrm{seq}_k(\mc{A})/\Es{}{k}$, $t\in S$, and $\bar{q}=(q_j)_{j<k}\in(\Znz)^k$. Then there is some $n \in \mathbb{N}$ such that there exist $s_i\in\Qnz$ and pairwise distinct $\bar{y}^i\in\varepsilon$ with $\tp_g(\bar{y}^i)=t$ for $i<n$ that satisfy:
    \begin{enumerate}
         \item We have $a=\sum_{i<n}s_i\sum_{j<k}q_jy^i_j \mod{\nBp\cap G^g}$.
         \item For all $i<n$:
            \begin{enumerate}
            \item $s_i\sum_{j<k}q_jy^i_j\in G^g$;
            \item $s_i\notin \Qp$;
            \item $s_i\sum_{j<k}q_jy^i_j\not\in \nBp\cap G^g$;
            \item each $s_iq_jy^i_j\not\in \nBp\cap G^g$ for $j<k$.
            \end{enumerate}
        \item If $n \geq 2$, then $\mathrm{supp}_p(a)=\mathrm{supp}_p(\sum_{i<n}s_i\sum_{j<k}q_jy^i_j)$ has size at least $2$.
    \end{enumerate}
\end{lemma}

\begin{proof} 

First we will find such a sequence satisfying (1) and (2). Given $a\in G^g$ we can write it as an integer linear combination of nonzero generators of $G^g$. In particular, we have some $n \in \mathbb{N}$ such that, for each $i<n$, there exist $\bar{y}^i\in \text{seq}_{k_i}(\mc{A})$, $r_i\in \Znz$, and $m_i\geq0$, such that, with $\varepsilon_i=\bar{y}^i/\Es{}{k_i}$, $t_i=\tp_g(\bar{y}^i)$, and $\bar{q}^i\in (\Znz)^{k_i}$ for which $p_i=p_{\varepsilon_i,t_i,\bar{q}^i}$, we have:
\begin{align*}
    a=\sum_{i<n}p_i^{-m_i}r_i\sum_{j<k_i}q^i_jy^i_j
\end{align*}
Since the $r_i$ and $q_j^i$ are in $\mathbb{Z}$, modulo $\mathbb{Z}_{(p)}\mc{A} \cap G^g$ we can remove all terms with $p_i \neq p$. Since the primes were listed without repetition, every remaining term has $\varepsilon_i=\varepsilon$, $t_i=t$, and $\bar{q}^i=\bar{q}$. Thus each remaining $\bar{y}^i\in\varepsilon$ satisfies $\tp_g(\bar{y}^i)=t$, and, after reindexing, we can write
\begin{align*}
	a=\sum_{i<n}p^{-m_i}r_i\sum_{j<k}q_j y^i_j \mod{\nBp\cap G^g}.
\end{align*}
If no tuple in $\varepsilon$ has type $t$ under $g$, this sum is empty and $n=0$. By grouping like terms and discarding zero coefficients, we may also assume that each tuple $\bar{y}^i$ appears only once (noting that a sum of terms of the form $p^{-m} r$ is again of that form). By cancelling factors of $p$, we may assume that the expressions $p^{-m_i} r_i$ are written in lowest terms, i.e., that if $m_i > 0$ then $p \nmid r_i$. Then in the sum we may assume that each $m_i > 0$ and $p \nmid r_i$, as otherwise $p^{-m_i}r_i\sum_{j<k}q_jy^i_j \in {\nBp\cap G^g}$.

Letting $s_i = p^{-m_i}r_i$, we have (1). We now check (2). For (a), since $\bar{y}^i\in\varepsilon$ and $\tp_g(\bar{y}^i)=t$, $p^{-m_i} \sum_{j<k}q_jy^i_j$ is a generator, and $r_i (p^{-m_i} \sum_{j<k}q_jy^i_j) = s_i \sum_{j<k}q_jy^i_j \in G^g$. For (b), since $m_i > 0$ and $p \nmid r_i$, $s_i = p^{-m_i} r_i \notin \mathbb{Z}_{(p)}$. For (c) and (d), note that $p = p_{\varepsilon,t,\bar{q}}$ was chosen so that for any $j$, $p \nmid q_j$. Thus as $s_i \notin \mathbb{Z}_{(p)}$, $s_i q_j \notin \mathbb{Z}_{(p)}$ and so $s_i q_j y_j^i \notin \mathbb{Z}_{(p)}\mc{A}$ and, as $y^i_j \neq y^i_{j'}$ for $j \neq j'$, $s_i \sum_{j < k} q_j y^i_j \notin \mathbb{Z}_{(p)}\mc{A}$.

Finally we must prove (3). Suppose in the sum
\begin{align*}
	a=\sum_{i<n}p^{-m_i}r_i\sum_{j<k}q_j y^i_j \mod{\nBp\cap G^g}.
\end{align*}
we have $n \geq 2$. Since the distinct tuples $\bar{y}^i$ all belong to the same $\mathbb{E}_k$-equivalence class $\varepsilon$, by $(3)$ of Theorem \ref{thm:eq-structure}, there are distinct $y^{i_1}_{j_1}, y^{i_2}_{j_2}$ such that:
    
\[ y^{i_1}_{j_1}  \notin \{y_j^i : \text{$(i_1,j_1) \neq (i,j)$}\}\]
and
\[ y^{i_2}_{j_2} \notin \{y_j^i : \text{$(i_2,j_2) \neq (i,j)$}\}\]
This means that in the sum above, we know that the terms $y^{i_1}_{j_1}$ and $y^{i_2}_{j_2}$ appear exactly once. Moreover, recalling that $s_iq_j\notin\Qp$ for all $i,j$, their coefficients will not be in $\mathbb{Z}_{(p)}$. It follows that $y^{i_1}_{j_1}$ and $y^{i_2}_{j_2}$ will be in the $p$-support of the sum, that is:
 
\[\{y^{i_1}_{j_1}, y^{i_2}_{j_2}\}\subseteq\text{supp}_p\left(\sum_{i<n}s_i\sum_{j<k}q_jy^i_j\right)\]
So  $\mathrm{supp}_p(a)=\text{supp}_p(\sum_{i<n}s_i\sum_{j<k}q_jy^i_j)$ has at least two elements, as desired.
\end{proof}

\subsection{The Isomorphism Relation}

In what follows, let $\n$ and $\m$ be models in $\Eq$, and let $g: \mc{A}/\mathbb{X}\rightarrow \n$ and $h: \mc{A}/\mathbb{X}\rightarrow\m$ be bijections. We will show that $\n\cong\m$ if and only if $G^g\cong G^h$. We begin with the easy direction.

\begin{lemma}\label{lem:graph-isom induces group-isom}
   Let $\n,\m\in \Eq_\ell$, and let $g:\mc{A}/\mathbb{X}\rightarrow \n$ and $h:\mc{A}/\mathbb{X}\rightarrow \m$ be bijections. Let $G^g$ and $G^h$ be the torsion-free abelian groups constructed from $\n,g$ and $\m,h$, respectively, as in Definition~\ref{def:group-structure}. If $\mc{M} \cong \mc{N}$ then $G^g \cong G^h$.
\end{lemma}

\begin{proof}
    Suppose that $\rho : \mc{N} \to \mc{M}$ is an isomorphism. Given the bijections $g:\mc{A}/\mathbb{X}\rightarrow\n$ and $h:\mc{A}/\mathbb{X}\rightarrow\m$, we get a bijection $\alpha=h^{-1}\circ\rho\circ g:\mc{A}/\mathbb{X}\rightarrow \mc{A}/\mathbb{X}$. By part $(2)$ of Theorem \ref{thm:eq-structure}, there is a strong automorphism $P: \mc{A} \rightarrow \mc{A}$ compatible with $\alpha$. So, if $x,y\in \mc{A}$ have $x\mathrel{\mathbb{X}}y$, then $P(x)\mathrel{\mathbb{X}}P(y)$. This $P$ induces an isomorphism  $\varrho:\bigoplus_{x\in \mc{A}}\mathbb{Q}x\xrightarrow{\sim}\bigoplus_{x\in \mc{A}}\mathbb{Q}x$ by permuting its generators: For $\sum_{i<n}q_ix_i\in \bigoplus_{x\in \mc{A}}\mathbb{Q}x$, let $\varrho(\sum_{i<n}q_ix_i)=\sum_{i<n}q_iP(x_i)$. We can define the inverse $\varrho^{-1}:\bigoplus_{x\in \mc{A}}\mathbb{Q}x\xrightarrow{\sim}\bigoplus_{x\in \mc{A}}\mathbb{Q}x$ similarly using $P^{-1}$. We will argue that these restrict to isomorphisms $G^g \to G^h$ and $G^h \to G^g$.
    
    It suffices to show that each generator of $G^g$ maps to a generator of $G^h$ (under $\varrho$) and that each generator of $G^h$ maps to a generator of $G^g$ (under $\varrho^{-1}$). The argument for the latter is symmetric to the argument for the former.

	First, consider a generator $x \in \mc{A}$ of $G^g$. This maps to the generator $P(x) \in \mc{A}$ of $G^h$. Second, consider a generator  $\sum_{i<n}p_{\varepsilon,t,\bar{q}}^{-m}q_ix_i$ of $G^g$, where $m \in \mathbb{N}$, $\bar{x}\in \text{seq}_n(\mc{A})$, $\varepsilon=\bar{x}/\Es{}{n}$, and $t=\tp_{g}(\bar{x})$. We must show that its image $\sum_{i<n}p_{\varepsilon,t,\bar{q}}^{-m}q_i P(x_i)$ is a generator of $G^h$. For this it is sufficient to show that  $P(\bar{x})=(P(x_i))_{i<n}$ is in the same $\mathbb{E}_n$ equivalence class $\varepsilon$ as $\bar{x}$, and $t=\tp_{g}(\bar{x})=\tp_{h}(P(\bar{x}))$. As $P$ is a strong automorphism of $\mc{A}$, it preserves the $\mathbb{E}_n$ equivalence classes, so $P(\bar{x})$ is in the same $\mathbb{E}_n$-equivalence class $\varepsilon$ as $\bar{x}$. That $\tp_{g}(\bar{x})=\tp_{h}(P(\bar{x}))$ follows from $\rho$ being an isomorphism and the following commutative diagram:
	 \[
	\xymatrix@C=4em@R=2.5em{
		\mc{A} \ar[r]^{\sim}_P \ar@{->>}[d] & \mc{A} \ar@{->>}[d] \\
		\mc{A}/\mathbb{X} \ar[r]^{\sim}_\alpha \ar[d]_g^{\rotatebox[origin=c]{90}{$\scriptstyle\sim$}} & \mc{A}/\mathbb{X} \ar[d]^h_{\rotatebox[origin=c]{90}{$\scriptstyle\sim$}} \\
		\mc{N} \ar[r]^{\sim}_\rho & \mc{M}
	}
	\]
	Recall that $\tp_{g}(\bar{x})$ and $\tp_{h}(P(\bar{x}))$ were defined using this diagram. We have now shown that each generator of $G^g$ maps by $\varrho$ to a generator of $G^h$, and a similar argument works in the other direction. Thus $\varrho$ induces an isomorphism $G^g \to G^h$.
\end{proof}

It is more involved to show that if $G^g\cong G^h$, then $\n\cong\m$ . First, we will show that an isomorphism on the groups sends the elements of $\mc{A}$ to rational multiples of other elements of $\mc{A}$, which will then allow us to define a bijection on $\mathcal{A}$. From there, we will argue this bijection induces an isomorphism on $\n$ and $\m$.  

%pi(x)=qy%

\begin{lemma}\label{lem:pi-on-group}
     Suppose $\pi:G^g\xrightarrow{\sim}G^h$ is an isomorphism, and let $x\in \mc{A}$ and $b=\pi(x)$. Then $b \in \mathbb{Q}y$ for some $y\in \mc{A}$.
\end{lemma}

\begin{proof}

    As $b\in G^h\subset\bigoplus_{x\in \mc{A}} \mathbb{Q}x$, for some $k \in \mathbb{N}$  there is a $\bar{y}\in\seq$, and a $\bar{q}\in (\Qnz)^k$ such that $b=\sum_{j<k}q_jy_j$. We want to show that $k=1$. To do this, let $q_* \in\Znz$ be the product of the denominators of the $q_j\in\Qnz$ so that, for every $j<k$, $q_*q_j\in\Znz$.% and if $p$ is a prime with $p\mid q_*$ then $p\mid q_*q_j$.

    Now, let $\bar{q}'=(q_*q_j:j<k)$, $\varepsilon=\bar{y}/\Es{}{k}$, $t=\tp_h(\bar{y})$, and $p=p_{\varepsilon,t,\bar{q}'}$. Then $p^{-m}\sum_{j<k}q_j'y_j$ is a generator of $G^h$ for all $m \in \mathbb{N}$. In particular, $q_*b=\sum_{j<k}q_j'y_j\in\tBp$, and---since $\tBp$ is a pure subgroup of $G^h$---we know $b\in \tBp$. Thus, as $b=\pi(x)$ and $\pi$ is an isomorphism, we must also have that $x\in \Bp$. 

    Then $p^{-1}x\in G^g$, and by Lemma \ref{lem:a-modulo-p} there is $n \in \mathbb{N}$ and $s_i\in\Qnz$ and pairwise distinct $\bar{z}^i\in \varepsilon$ with $\tp_g(\bar{z}^i)=t$ for $i < n$ such that
    \begin{align*}
        p^{-1}x=\sum_{i<n}s_i\sum_{j<k}q'_jz^i_j \mod{\nBp\cap G^g}
    \end{align*}
with $s_i \notin \Qp$, $s_i q_j' \notin \Qp$, and $s_i\sum_{j<k}q_j'z^i_j \neq0\mod{\nBp\cap G^g}$ for every $i<n$.

 Consider the  $p$-support of $p^{-1} x$. Clearly $p^{-1} x$ has singleton $p$-support, so $\sum_{i<n}s_i\sum_{j<k}q_j'z^i_j$ must as well. This is only possible if $n=1$:
    \begin{itemize}[leftmargin=1.5cm]
        \item[$n \geq 2$:] By part $(3)$ of Lemma \ref{lem:a-modulo-p},  if $n\geq2$, then $\mathrm{supp}_p(\sum_{i<n}s_i\sum_{j<k}q_j'z^i_j)$ has at least two elements. But this is a contradiction, as the $p$-support should be a singleton.
        \item[$n=0$:] If $n=0$, then $p^{-1}x=0\mod{\nBp\cap G^g}$. But $p^{-1}x\notin\nBp$, so this is a contradiction.
    \end{itemize}
    Therefore, $n=1$. Thus we have $k=1$. Otherwise, as $\bar{z}^0\in\seq$ has no repetition, $q_j'\in\Znz$, and $s_0q_j'\notin\Qp$,  we would have $|\mathrm{supp}_p(s_0\sum_{j<k}q'_jz^0_j)|=k$. This is a contradiction. Therefore, $k=1$, and so $b=q_0y_0\in\mathbb{Q}y_0$ as desired.
  \end{proof}

\begin{lemma}\label{lem:pi-on-comb}
    An isomorphism $\pi :G^g\xrightarrow{\sim}G^h$ induces a bijection $\pi_*:\mc{A}\rightarrow\mc{A}$.
  
\end{lemma}
\begin{proof}
    Define the map $\pi_*:\mc{A}\rightarrow\mc{A}$ induced by $\pi$ as follows: for $x\in \mc{A}$, if for some $y\in \mc{A}$ $\pi(x)\in\mathbb{Q}y$, then $\pi_*(x)=y$. By Lemma \ref{lem:pi-on-group}, such a $y$ exists for every $x\in \mc{A}$; it is unique because $\pi(x)\neq0$ and the elements of $\mc{A}$ are linearly independent. Similarly, applying Lemma \ref{lem:pi-on-group} to $\pi^{-1}$, we can define the map $\pi_*^{-1}:\mc{A}\rightarrow\mc{A}$ induced by $\pi^{-1}$ as follows: for $y\in \mc{A}$, if for some $x\in \mc{A}$ $\pi^{-1}(y)\in\mathbb{Q}x$, then $\pi_*^{-1}(y)=x$. We must show that $\pi_*$ and $\pi_*^{-1}$ are inverses.

    Let $x\in \mc{A}$ and $y=\pi_*(x)$. Then $\pi(x)=qy$ for some $q\in\Qnz$. Since $\pi^{-1}$ commutes with rational scalar multiplication, we have
    \[
        \pi^{-1}(y)=\pi^{-1}(q^{-1}\pi(x))=q^{-1}x\in\mathbb{Q}x.
    \]
    It follows that $\pi_*^{-1}(\pi_*(x))=\pi_*^{-1}(y)=x$. Applying the same argument to $\pi^{-1}$ gives $\pi_*(\pi_*^{-1}(y))=y$ for every $y\in \mc{A}$. Hence $\pi_*$ and $\pi_*^{-1}$ are inverses, and $\pi_*$ is a bijection.
\end{proof}

Now that we know an isomorphism $\pi:G^g\xrightarrow{\sim} G^h$ induces a bijection $\pi_*:\mc{A}\rightarrow\mc{A}$, we would like to use it to induce an isomorphism from $\n$ to $\m$. Recall that $\n$ and $\m$ are in bijection with the quotient $\mc{A} / \mathbb{X}$ via $g$ and $h$ respectively. So first we will want to show that $\pi_*$ descends to a bijection on the quotient, that is, $x\mathrel{\mathbb{X}}y \iff \pi_*(x)\mathrel{\mathbb{X}}\pi_*(y)$. This will allow $\pi_*$ to induce a bijection $\n \to \m$. To make this an isomorphism, we must show that it preserves the structure of $\n$ and $\m$. Since the languages of $\n$ and $\m$ consist only of binary relations, it will suffice to show that for $x,y \in \mc{A}$, $\tp_{g}( (x,y))=\tp_{h}((\pi_*(x),\pi_*(y)))$.

\begin{lemma}\label{lem:pi-preserves-types}
    Let  $\pi_*:\mc{A}\rightarrow\mc{A}$ be the bijection induced by the isomorphism $\pi :G^g\xrightarrow{\sim}G^h$. For $x,y\in \mc{A}$, the following hold:
    \begin{enumerate}
        \item $\tp_{g}( (x,y))=\tp_{h}((\pi_*(x),\pi_*(y)))$; 
        \item $x\mathrel{\mathbb{X}}y$ if and only if $\pi_*(x)\mathrel{\mathbb{X}}\pi_*(y)$.
    \end{enumerate}
\end{lemma}

\begin{proof}

    Here, $(1)$ amounts to showing that for $(x_1,x_2)\in \mc{A}^2$, $g(x_1,x_2)\in \n^2$ and $h(\pi_*(x_1),\pi_*(x_2))\in \m^2$ have the same atomic type. Fix some $\ell \in \mathbb{N}$, and recall that the vocabulary of $\Eq_{\ell}$ is $(\rE_1,\cdots,\rE_{\ell})$, where each $E_d$ for $1\leq d\leq \ell$ is an equivalence relation. For convenience, let $E_0$ denote equality. For any $2$-tuple $(x_1,x_2)\in \mc{A}^2$ it is sufficient to show that $g(x_1)\rE_dg(x_2)$ if and only if  $h(\pi_*(x_1))\rE_d h(\pi_*(x_2))$ for every $0\leq d\leq \ell$, as all relations on $\Eq_\ell$ are binary. From this we also get $(2)$, as $g(x_1)\rE_0g(x_2)$ if and only if $x_1\mathrel{\mathbb{X}}x_2$, and $h(x_1)\rE_0h(x_2)$ if and only if $x_1\mathrel{\mathbb{X}}x_2$.

    Fix $\bar{x}=(x_1,x_2)\in \mc{A}^2$ and $0 \leq d \leq \ell$. We will show that $g(x_1)\rE_dg(x_2)$ if and only if  $h(\pi_*(x_1))\rE_d h(\pi_*(x_2))$. We may assume that $x_1\neq x_2$ as if $x_1=x_2$, then $\pi_*(x_1)=\pi_*(x_2)$ and the conclusion is clear.
    
    We begin by assuming that $g(x_1)\rE_dg(x_2)$ and must show that $h(\pi_*(x_1))\rE_d h(\pi_*(x_2))$.  Lemma \ref{lem:pi-on-group} tells us there are $y_1,y_2\in \mc{A}$ and $q_1,q_2\in\Qnz$
    for which $\pi(x_1)=q_1y_1$ and $\pi(x_2)=q_2y_2$. Let $\varepsilon=\bar{x}/\Es{}{d}$, $t=\tp_g(\bar{x})$, and $p=p_{\varepsilon,t,(1,1)}$. As $x_1+x_2\in\Bp$ and $\pi$ is an isomorphism, $q_1y_1+q_2y_2\in\tBp$. Then there is some $m \in \mathbb{N}$ such that $p^{-m}q_1\notin\Qp$ and $p^{-m}q_2\notin\Qp$. By Lemma
    \ref{lem:a-modulo-p}, for some $n \in \mathbb{N}$, for all $i<n$ there are pairwise distinct $\bar{w}^i\in\varepsilon$ with $\tp_h(\bar{w}^i)=t$ and $s_i\not\in\Qp$ such that:
    \begin{align*}
        p^{-m}q_1y_1+p^{-m}q_2y_2=\sum_{i<n}s_i(w_1^i+w_2^i)\mod\nBp\cap G^h .
    \end{align*}
    Since $\tp_h(\bar{w}^i)=t=\tp_g(\bar{x})$, for every $i<n$ we have $h(w^i_1)\rE_dh(w^i_2)$ as $g(x_1)\rE_dg(x_2)$. We can put an equivalence relation on the indices $\{0,\ldots,n-1\}$ by putting $i \sim j$ if $h(w^i_1) \rE_d h(w^j_1)$ (which is equivalent to, e.g., $h(w^i_1) \rE_d h(w^j_2)$). Let $k$ be the number of equivalence classes, and let $U_{k'}$ list the equivalence classes for $k' < k$. For each such $k'$, the elements
    \[ \{h(w_1^i) : i \in U_{k'}\} \cup \{ h(w_2^i) : i \in U_{k'} \} \]
    are all $E_d$-equivalent. We can reorganize the sum as
    \begin{align*}
        p^{-m}q_1y_1+p^{-m}q_2y_2=\sum_{k'< k}\sum_{i\in U_{k'}}s_i(w_1^i+w_2^i)\mod\nBp\cap G^h .
    \end{align*}
    By our choice of $m$ and since $y_1\neq y_2$, we have $\mathrm{supp}_p(p^{-m}q_1y_1+p^{-m}q_2y_2)=\{y_1,y_2\}$. Also, $\mathrm{supp}_p(\sum_{i\in U_{k'}}s_i(w_1^i+w_2^i))$ has at least two elements for every $k'<k$. Indeed, if $U_{k'}=\{i\}$, then $w_1^i\neq w_2^i$ and both coordinates have coefficient $s_i\notin\Qp$. If $|U_{k'}|\geq2$, then part $(3)$ of Theorem \ref{thm:eq-structure} gives two coordinates each occurring exactly once among the distinct tuples $\bar{w}^i$, $i\in U_{k'}$, in $\varepsilon$. Their coefficients are among the $s_i\notin\Qp$, so both coordinates belong to the $p$-support.
    
    For distinct $k_1,k_2 < k$, the sets
    \[ \{ w_1^i,w_2^i : i \in U_{k_1}\} \text { and } \{ w_1^i,w_2^i : i \in U_{k_2}\} \]
    are disjoint, since their images under $h$ lie in distinct $E_d$-equivalence classes. So $\mathrm{supp}_p(\sum_{i\in U_{k_1}}s_i(w_1^i+w_2^i))$ and $\mathrm{supp}_p(\sum_{i\in U_{k_2}}s_i(w_1^i+w_2^i))$ are disjoint. Thus:
    \begin{align*}
        \left|\mathrm{supp}_p\left(\sum_{k'<k}\sum_{i\in U_{k'}}s_i(w_1^i+w_2^i)\right)\right| =\sum_{k'<k}\left|\mathrm{supp}_p\left(\sum_{i\in U_{k'}}s_i(w_1^i+w_2^i)\right)\right|\geq 2k.
    \end{align*}
    Since $\mathrm{supp}_p(p^{-m}q_1y_1+p^{-m}q_2y_2)$ is $\{y_1,y_2\}$, we must have that $k=1$. Moreover, we find $\{y_1,y_2\}=\mathrm{supp}_p(\sum_{i\in U_0}s_i(w_1^i+w_2^i))$, so $\{y_1,y_2\}=\{w_{j_1}^{i_1},w_{j_2}^{i_2}\}$ for some $i_1,i_2\in U_0$ and $j_1,j_2\in\{1,2\}$. Since $i_1,i_2\in U_0$, we have $h(w_{j_1}^{i_1})\rE_d h(w_{j_2}^{i_2})$. This implies that $h(y_1)\rE_dh(y_2)$ as desired. To show that if $h(\pi_*(x_1))\rE_d h(\pi_*(x_2))$, then $g(x_1)\rE_d g(x_2)$, we use the same argument with $\pi^{-1}$. This completes the proof.
    \end{proof}

\begin{lemma}\label{lem:group-isom induces graph-isom}
     Let $\n,\m\in \Eq$, and given any choice of bijections $g:\mc{A}/\mathbb{X}\rightarrow \n$ and $h:\mc{A}/\mathbb{X}\rightarrow \m$, if $G^g\cong G^h$, then $\n\cong\m$.
\end{lemma}

\begin{proof}
    Let $\pi:G^g\xrightarrow{\sim}G^h$ be an isomorphism. Then by Lemmas~\ref{lem:pi-on-comb} and~\ref{lem:pi-preserves-types}, $\pi$ induces a bijective map $\pi_*:\mc{A}\rightarrow \mc{A}$ such that for $x,y\in \mc{A}$, $\tp_{h}((\pi_*(x),\pi_*(y)))=\tp_{g}( (x,y))$ and 
    $x\mathrel{\mathbb{X}}y$ if and only if $\pi_*(x)\mathrel{\mathbb{X}}\pi_*(y)$. The latter tells us that we have an induced bijection  $\pi_{**}:\mc{A}/\mathbb{X}\rightarrow \mc{A}/\mathbb{X}$. Let $\rho=h \circ \pi_{**} \circ g^{-1}: \n\rightarrow \m$. We claim $\rho$ is an isomorphism. 

    Clearly, $\rho$ is a bijection as it is a composition of bijections. Thus, to show $\rho$ is an isomorphism it is sufficient to show that it preserves the atomic diagrams of $2$-tuples, as all relations on $\n$ and $\m$ are binary. Let $u,v\in\n$.  Recall Definition \ref{def:atomic-types}: for any $x\in g^{-1}(u)$ and $y\in g^{-1}(v)$ $\tp_{g}(x,y)$ is the atomic type of $(u,v)$. By Lemma \ref{lem:pi-preserves-types}, $x\in g^{-1}(u)$ and $y\in g^{-1}(v)$ implies $\tp_{g}(x,y)=\tp_{h}(\pi_*(x),\pi_*(y))$, and $\pi_*(x)\in \pi_{**}(g^{-1}(u))$, $\pi_*(y)\in \pi_{**}(g^{-1}(v))$. Recall that $\tp_{h}(\pi_*(x),\pi_*(y))$ is the atomic type of $(h(\pi_{**}(g^{-1}(u))),h(\pi_{**}(g^{-1}(v))))=(\rho(u),\rho(v))$. So $\rho$ preserves the atomic types of $2$-tuples, and is therefore an isomorphism.     
\end{proof}

We now have a Borel reduction from $\Eq_2$ to the class of torsion-free abelian groups: Given an presentation $\mc{N}$ with domain $\mathbb{N}$, and taking a bijection $g : \mc{A} / \mathbb{X}\to \mc{N}$ using the fixed indexing described after Definition~\ref{def:group-structure}, we construct a presentation of $G(\mc{N}) \cong G^{g}$ with domain $\mathbb{N}$. By Lemmas~\ref{lem:graph-isom induces group-isom} and~\ref{lem:group-isom induces graph-isom}, we have
\[ \mc{N} \cong \mc{M} \Longleftrightarrow G(\mc{N}) \cong G(\mc{M}).\]
Since $\Eq_2$ is Borel complete by Section~\ref{sec:two-equivalence-relations}, the class of torsion-free abelian groups is Borel complete.

\section{Coding a family into a torsion-free abelian group}\label{sec:slaman-wehner}

In this section we will prove Theorem~\ref{thm:slaman wehner} by showing that there is a torsion-free abelian group which has no computable copy, but has a copy computable from any non-computable set. Like with many such arguments, we will use enumerations of families. By a family, we mean an infinite collection of subsets of $\mathbb{N}$.

\begin{definition}
An \emph{enumeration} of a family $\mathcal{F}$ of subsets of $\N$
is a set $W\subseteq\N\times\N$ such that
\[
    \{W^{[n]}:n\in\N\}=\mathcal{F},
    \qquad W^{[n]}=\{m:(n,m)\in W\}.
\]
We will take our enumerations to have infinitely many repetitions of each set.
We say that an oracle $X$ can \emph{enumerate $\mathcal{F}$} if
$\mathcal{F}$ has an $X$-c.e.\ enumeration.
\end{definition} 

There are some obvious steps to try coding a family of sets into a torsion-free abelian group. Fixing a computable list of distinct primes $r_n$, given an element $x$ of the group and a set $S$, we may make $x$ divisible by $r_n$ if and only if $n \in S$. Thus, given $x$, we can enumerate $S$ relative to the atomic diagram of the group. On the other hand, $r_n x$ encodes $S \cup \{n\}$, and if $n \in S$ and $r_n^2\nmid x$, then $r_n^{-1}x$ encodes $S\setminus\{n\}$. One possible strategy is to account for this in the family, building, e.g., a family that is closed under finite differences. This has two issues, one being that we do not know whether there is a Slaman-Wehner family closed under finite differences (see \cite{GHTMT,Faizrahmanov}), and the other being that there are also sums $x+y$ of different elements $x$ and $y$ coding totally different sets from each other. We do not follow this strategy. Instead, construct our group inside a rational vector space with a distinguished basis and code sets only using these basis elements. The basis elements cannot be defined by a computable infinitary $\Sigma_1$ formula (even with parameters), but we make them definable by a computable infinitary $\Pi_1$ formula over a single parameter. This is the crux of the argument.  Given this parameter, we can enumerate the non-basis elements relative to the atomic diagram of the group. This allows us to adapt the Slaman-Wehner family to our construction.

\subsection{Coding a family into a group}

Given an enumeration $W$ of a family $\mc{F}$, we will construct a torsion-free abelian group $G_{\mc{F}}^W$ which encodes the family. We will use the machinery of Section~\ref{sec:borel-completeness-tfag}. We are no longer encoding atomic types from models in $\Eq_\ell$. This is the same as working with $\ell = 0$, as in that case there is only one atomic type of each arity. We thus drop the type $t$ from the list of primes. In addition to the list of primes $p$, we will also need a new list of primes $r$. These lists will need to be computable.

Fix a computable presentation of the structure $\mathcal{A}$
constructed in Theorem~\ref{thm:eq-structure}. Such a presentation
is obtained by carrying out its finite-stage construction effectively.

\begin{definition}\label{def:family-primes}
Fix two computable lists of primes, with no repetitions within
or between the lists.
The first list consists of pairwise distinct primes $r_n$ for $n \in \mathbb{N}$ and the second of $p_{\varepsilon,\bar{q}}$ for $n \in \mathbb{N}$, $\varepsilon\in\mathrm{seq}_n(\mc{A})/\mathbb{E}_n$, and $ \bar{q}\in(\Znz)^n$.
The second list is chosen such that $p_{\varepsilon,\bar{q}} \nmid q_j$ and also to satisfy requirements of Remark~\ref{rem:p-Zsigmondy} below.
\end{definition}

We define our torsion-free abelian groups as follows.

\begin{definition}\label{def:bouquet group}
     Fix some enumeration of $\mc{A}/\mathbb{X}$. Let $\mathcal{F}$ be a countable family, and let $W$ be an enumeration of it.
     \begin{enumerate}
         \item Let $\alpha_\mathcal{F}^W:\mc{A}/\mathbb{X}\rightarrow \mathcal{F}$ be a bijection mapping $\mathbb{X}$-equivalence classes to sets in $\mc{F}$.
         \item  Define $G^W_\mathcal{F}$ to be the subgroup of $\bigoplus_{x\in \mc{A}}\mathbb{Q}x$ generated by:
         \begin{itemize}
            \item $\bigoplus_{x\in \mc{A}}\mathbb{Z}x$;
            \item $P=\{p_{\varepsilon,\bar{q}}^{-m}\sum_{j<n} q_jy_j:m,n \in \mathbb{N},\bar{q}\in(\Znz)^n,\bar{y}\in \mathrm{seq}_n(\mc{A}), \text{ and } \varepsilon=\bar{y}/\mathbb{E}_n\}$;
            \item $R_\mathcal{F}^W=\{r_{n}^{-1}x:x\in \mc{A},n\in\alpha_\mathcal{F}^W(x/\mathbb{X})\}$.
    \end{itemize}
     \end{enumerate}
 	 When it is clear what family $\mathcal{F}$ we are talking about, we will write $\alpha^W=\alpha^W_{\mathcal{F}}$, and similarly for $G^W=G^W_\mathcal{F}$, $R^W=R^W_\mathcal{F}$. We will refer to the elements of $\mc{A}$ as the \emph{basis elements} of $G^W_\mathcal{F}$.
\end{definition}

For any countable family $\mathcal{F}$ and any set $W$ which enumerates $\mathcal{F}$, Lemmas \ref{lem:family-a-modulo-p} and \ref{lem:family-pi-on-group} give the corresponding results for $G^W_\mathcal{F}$. The same proofs as in Section \ref{sec:borel-completeness-tfag} apply: for each coding prime $p=p_{\varepsilon,\bar{q}}$, the additional generators in $R^W$ lie in $\nBp\cap G^W$, since the primes $r_n$ are distinct from $p$, and so these generators disappear modulo $\nBp\cap G^W$. The primes $p_{\varepsilon,\bar{q}}$ are now indexed only by the $\mathbb{E}_k$-equivalence class and the coefficient tuple.

\begin{lemma}
	\label{lem:family-a-modulo-p}
	Let $W$ be an enumeration of a countable family $\mathcal{F}$. Let $a\in G^W$, and let $p=p_{\varepsilon,\bar{q}}$ where $\varepsilon\in\mathrm{seq}_k(\mc{A})/\Es{}{k}$ and $\bar{q}=(q_j)_{j<k}\in(\Znz)^k$. Then there is some $n \in \mathbb{N}$ such that there exist $s_i\in\Qnz$ and pairwise distinct $\bar{y}^i\in\varepsilon$ for $i<n$ that satisfy:
	\begin{enumerate}
		\item We have $a=\sum_{i<n}s_i\sum_{j<k}q_jy^i_j \mod{\nBp\cap G^W}$.
		\item For all $i<n$:
		\begin{enumerate}
			\item $s_i\sum_{j<k}q_jy^i_j\in G^W$;
			\item $s_i\notin \Qp$;
			\item $s_i\sum_{j<k}q_jy^i_j\not\in \nBp\cap G^W$;
			\item each $s_iq_jy^i_j\not\in \nBp\cap G^W$ for $j<k$.
		\end{enumerate}
		\item If $n \geq 2$, then $\mathrm{supp}_p(a)=\mathrm{supp}_p(\sum_{i<n}s_i\sum_{j<k}q_jy^i_j)$ has size at least $2$.
	\end{enumerate}
\end{lemma}

\begin{lemma}\label{lem:family-pi-on-group}
	Let $W$ and $V$ be enumerations of countable families $\mathcal{F}$ and $\mathcal{H}$, respectively. Suppose $\pi:G^W_{\mathcal{F}}\xrightarrow{\sim}G^V_{\mathcal{H}}$ is an isomorphism, and let $x\in \mc{A}$ and $b=\pi(x)$. Then $b \in \mathbb{Q}y$ for some $y\in \mc{A}$.
\end{lemma}

We also have a new lemma which described $r_n$-divisibility.

\begin{lemma}\label{lem:r divisors}
    Let  $\mathcal{F}$ be a countable family, and let $W$ be an enumeration of it. For any basis element $x\in \mc{A} \subseteq G^W$, $r_n\mid x$ if and only if $r_n^{-1}x\in R^W$ if and only if $n\in\alpha^W(x/\mathbb{X})$.
\end{lemma}
\begin{proof}
    The second equivalence and the reverse implication in the first
    follow from the definition of $R^W$. Suppose $x \in \mc{A} \subseteq G^W$ is
    a basis element and that there is some $n\in\N$ such that $r_n\mid x$.
    Then we can write $r_n^{-1}x$ as an integer linear combination of
    generators of $G^W$. Every generator in
    $\bigoplus_{y\in \mc{A}}\mathbb{Z}y$ or $P$ has all coefficients
    in $\Q{r_n}$. Among the generators in $R^W$, the only possible
    generator whose coefficient of $x$ lies outside $\Q{r_n}$ is
    $r_n^{-1}x$, because the primes $r_m$ are pairwise distinct.
    This generator must occur, hence $r_n^{-1}x\in R^W$.
\end{proof}

\begin{lemma}\label{lem:enumeration isomorphism types}
    Let $\mathcal{F}$ be a countable family of sets, and let  $W,\tilde{W}$ be enumerations of $\mathcal{F}$. Then $G^W\cong G^{\tilde{W}}$.
\end{lemma}
\begin{proof}
    Let $\mathcal{F}$ be a countable family of sets, and let $\tilde{W},W\subset\N\times\N$ be enumerations of $\mathcal{F}$. Let $\alpha^{\tilde{W}}:\mc{A}/\mathbb{X}\rightarrow\mathcal{F}$ and $\alpha^W:\mc{A}/\mathbb{X}\rightarrow\mathcal{F}$ be bijections as defined in Definition \ref{def:bouquet group}. Then $(\alpha^{\tilde{W}})^{-1}\circ\alpha^W:\mc{A}/\mathbb{X}\rightarrow \mc{A}/\mathbb{X}$ is a bijection on the $\mc{A}/\mathbb{X}$. By Theorem \ref{thm:eq-structure}, there is a strong automorphism $\beta$ on $\mc{A}$ that is compatible with $(\alpha^{\tilde{W}})^{-1}\circ\alpha^W$. We can use $\beta$ to induce an automorphism $\gamma$ on $\bigoplus_{x\in \mc{A}}\mathbb{Q}x$ by permuting its generators. Explicitly: For every $n \in \mathbb{N}$, given $\bar{q}\in\mathbb{Q}^n$ and $\bar{x}\in \mc{A}^n$, let $\gamma(\sum_{i<n}q_ix_i)=\sum_{i<n}q_i\beta(x_i)$. We want to show that $\gamma(G^W)=G^{\tilde{W}}$.

    To show that $\gamma(G^W)\subset G^{\tilde{W}}$ it is enough to show that $\gamma$ maps the generators of $G^W$ to the generators of $G^{\tilde{W}}$. As $\beta$ is a strong automorphism of $\mc{A}$, it acts as an automorphism on the subgroup of $\bigoplus_{x\in \mc{A}}\mathbb{Q}x$ generated by $\bigoplus_{x\in \mc{A}}\mathbb{Z}x$ and $P$.   So all we need to show is that $\gamma(R^W)\subset R^{\tilde{W}}$.    

    Suppose $r_{n}^{-1}x\in R^W$. Then $n\in \alpha^W(x/\mathbb{X})$. We know that  $\beta$ is compatible with $(\alpha^{\tilde{W}})^{-1}\circ\alpha^W$, so $\beta(x)\in (\alpha^{\tilde{W}})^{-1}\circ\alpha^W(x/\mathbb{X})$. Then $\alpha^{\tilde{W}}(\beta(x)/\mathbb{X})=\alpha^W(x/\mathbb{X})$, so $\gamma(r_{n}^{-1}x)=r_{n}^{-1}\beta(x)\in R^{\tilde{W}}$, as $n\in \alpha^W(x/\mathbb{X})=\alpha^{\tilde{W}}(\beta(x)/\mathbb{X})$.  Therefore, $\gamma(R^W)\subset R^{\tilde{W}}$.

    Since $\beta$ is a bijection, it follows that we can define the automorphism $\gamma^{-1}$ of $\bigoplus_{x\in \mc{A}}\mathbb{Q}x$ as the inverse of $\gamma$ above. So  $\gamma^{-1}(R^{\tilde{W}})\subset R^W$ will follow by an analogous argument on $\gamma^{-1}$. Thus, $\gamma|_{G^W}:G^W\xrightarrow{\sim} G^{\tilde{W}}$ is an isomorphism.
\end{proof}

Therefore, for any countable family $\mathcal{F}$ we will denote the isomorphism type  $G_\mathcal{F}=G_\mathcal{F}^W$, where $W$ is any enumeration of $\mathcal{F}$.

The aim is to show that by fixing a basis element  $a\in G$, we can find a  $\Pi_1$ formula $\varphi$ such that for $x\in G$, $G\models \varphi(a,x)$ if and only if $x$ is also a basis element of $G$.  We will identify whether $x\in G$ is a generator based on which primes $p$ it is divisible by. To do this, we will make two claims. First,  for any $n>1$, $x\in G$  is not divisible by any   $p_{\varepsilon,\bar{q}}$ such that $\varepsilon\in\mathrm{seq}_n(\mc{A})/\mathbb{E}_n$ and $\bar{q}\in(\Znz)^n$. This will tell us that $x$ is not the sum of many basis elements of $G$. Second,  $x+a$ is not divisible by any $p_{\varepsilon,\bar{q}}$ such that $\varepsilon\in\mathrm{seq}_2(\mc{A})/\mathbb{E}_2$ and $\bar{q}\in(\Znz)^2$, except those $p_{\varepsilon,q}$ like $\varepsilon = (a,x)/\mathbb{E}_2$ and $q = (1,1)$ for which there is an obvious divisibility. This implies that $x$ is not a multiple of a basis element of $G$.

However,  for $\varepsilon\in\mathrm{seq}_2(\mc{A})/\mathbb{E}_2$ and $\bar{q}\in(\Znz)^2$ we must pick $p_{\varepsilon,\bar{q}}$ carefully, otherwise the second claim is not necessarily true. Suppose $(a,y),(y,x)\in \varepsilon$, $\bar{q}=(1,5)$, and $p_{\varepsilon,\bar{q}}=13$. Then $a+x=13(\frac{a+5y}{13}-5\frac{y+5x}{13}+2x)$, so $a+x$ is divisible by $13$ and the second claim fails. To avoid this, for each coprime pair $(q_1,q_2)$ with $q_1 \neq q_2$, we will pick $p_{\varepsilon,\bar{q}}$ such that $(-q_1q_2^{-1})^k\neq -1\mod{p}$. As the following Lemma shows, for any such pair there are infinitely many such primes $p$ we may choose from.

\begin{lemma}\label{lem:Zsigmondy}
    For each pair of distinct $q_1,q_2\in\Znz$ such that $\gcd(q_1,q_2)=1$ and $q_1>0$, there are infinitely many odd primes $p\nmid q_1q_2$ such that, for all $k\in\mathbb{Z}$, $(-q_1q_2^{-1})^k\not\equiv-1\mod{p}$.
\end{lemma}

\begin{proof}
   To show this, we can apply Zsigmondy's Theorem, which states that for every $n>6$ and distinct $q_1,q_2\in\N$ such that $\gcd(q_1,q_2)=1$, there exists a prime $p_n$ such that $p_n|q_1^n-q_2^n$ and $p_n\nmid q_1^j-q_2^j$ for every $1 \leq j <n$. The theorem gives a similar result for sums, stating that for every $n>6$ and distinct $q_1,q_2\in\N$ such that $\gcd(q_1,q_2)=1$, there exists a prime $p_n$ such that $p_n|q_1^n+q_2^n$ and $p_n\nmid q_1^j+q_2^j$ for every $1 \leq j<n$.  To apply Zsigmondy's Theorem, we must treat the cases where $q_2<0$ and $q_2>0$ separately.

    If $q_2<0$, we first consider the case $q_1 = 1$ and $q_2 = -1$. Then $(-q_1 q_2^{-1})^k = 1$ for any $k$, and so we can choose any odd prime $p$. So suppose $(q_1,q_2)\neq(1,-1)$. Let $\tilde{q}_2=-q_2$. Since $\gcd(q_1,\tilde{q}_2)=1$, we have $\tilde{q}_2\neq q_1$. We will show that we can pick infinitely many primes $p$ such that $\tilde{q}_2q_1^{-1}$ has odd multiplicative order in $(\mathbb{Z}/p\mathbb{Z})^\times$. This will tell us that   for all $k\in \N$, $(\tilde{q}_2q_1^{-1})^k\not\equiv -1\mod{p}$, as otherwise $\tilde{q}_2q_1^{-1}$ would have even order. For every odd $n>6$ and any prime $p$, $\tilde{q}_2^nq_1^{-n}\equiv1\mod{p}$ is equivalent to $\tilde{q}_2^n-q_1^n\equiv0\mod {p}$. By Zsigmondy's theorem, using the fact that $\tilde{q}_2\neq q_1$, there exists a prime $p_n$ such that $p_n|\tilde{q}_2^n-q_1^n$ and  $p_n\nmid \tilde{q}_2^j-q_1^j$ for every $1\leq j<n$, which implies that $n$ is the order of $(-q_2)q_1^{-1}$ modulo $p_n$. Thus for every $k\in\mathbb{Z}$,  $(-q_1q_2^{-1})^k\not\equiv -1\mod{p_n}$.
    
    If $q_2>0$, then for every odd $n>6$, $(-q_2)^nq_1^{-n}\equiv1\mod{p}$ is equivalent to $q_2^n+q_1^n\equiv0\mod {p}$ for $p$ a prime. By Zsigmondy's Theorem for sums, there exists a prime $p_n$ such that $p_n|q_1^n+q_2^n$ and $p_n\nmid q_1^j+q_2^j$ for every $1\leq j<n$. Such a prime is odd and does not divide $q_1q_2$. Since $n$ is odd, we have
    \[
        (-q_1q_2^{-1})^n=-q_1^nq_2^{-n}\equiv1\mod{p_n}.
    \]
    Thus the multiplicative order of $-q_1q_2^{-1}$ modulo $p_n$ divides $n$ and is therefore odd. Since $-1$ has order $2$ modulo the odd prime $p_n$, no power of $-q_1q_2^{-1}$ can equal $-1$. Hence, for every $k\in\mathbb{Z}$, $(-q_1q_2^{-1})^k\not\equiv-1\mod{p_n}$.

    As there are infinitely many odd $n>6$, we can pick infinitely many such primes for each $q_1,q_2\in\Znz$, as desired.
\end{proof}

\begin{remark}\label{rem:p-Zsigmondy}
    Consider the listing of primes given in Definition \ref{def:family-primes}. The definition specified that the primes $p_{\varepsilon,\bar{q}}$ would satisfy a further condition to be specified later, and we are now ready to specify this. For each $\varepsilon\in\mathrm{seq}_2(\mc{A})/\mathbb{E}_2$ and $\bar{q}=(q_1,q_2)\in\Znz^2$ distinct such that $\gcd(q_1,q_2)=1$ and $q_1>0$, by Lemma \ref{lem:Zsigmondy}, we must choose $p=p_{\varepsilon,\bar{q}}>6$ to ensure that $(-q_2q_1^{-1})^k\not\equiv-1\mod{p}$ for every $k\in\N$.
\end{remark}

\begin{lemma}\label{lem:x is not a sum}
    Suppose $x\in \mc{A}$ and let $\varepsilon\in\mathrm{seq}_n(\mc{A})/\mathbb{E}_n$ and $\bar{q}\in(\Znz)^n$ for $n>1$. Then $p_{\varepsilon,\bar{q}} \nmid x$.
\end{lemma}

\begin{proof}
     Let $x\in \mc{A}$ and $\varepsilon\in\mathrm{seq}_n(\mc{A})/\mathbb{E}_n$ and $\bar{q}\in(\Znz)^n$ for $n>1$, and let $p=p_{\varepsilon,\bar{q}}$. Suppose for the sake of contradiction that  $p\mid x$. Then by Lemma \ref{lem:family-a-modulo-p}, for some $m \in \mathbb{N}$, for $i<m$ there exists $s_i\notin \Q{p}$ and $\bar{y}^i\in\varepsilon$, such that $p^{-1}x=\sum_{i<m}s_i\sum_{j<n}q_jy^i_j\mod{\mathbb{Z}_{(p)}\mc{A}\cap G}$. But $p^{-1}x$ has singleton $p$-support, whereas $\sum_{i<m}s_i\sum_{j<n}q_jy^i_j\mod{\mathbb{Z}_{(p)}\mc{A}\cap G}$ has $p$-support of at least $2$ by Lemma \ref{lem:family-a-modulo-p} unless $m = 1$. But if $m = 1$, then the $p$-support of the right-hand-side contains $\{y_j^0 : j < n\}$ and $n > 1$. This is a contradiction, so $p\nmid x$.
\end{proof}

\begin{lemma} \label{lem:x is not 2x}
    Let $x,a\in \mc{A}$ be distinct and let $\bar{q}=(q_1,q_2)\in\Znz^2$ with $\gcd(q_1,q_2)=1$, $q_1>0$, and $q_1\neq q_2$. Given $\varepsilon\in\mathrm{seq}_2(\mc{A})/\mathbb{E}_2$, $p_{\varepsilon,\bar{q}} \nmid a+x$.
\end{lemma}

\begin{proof}
Let $\varepsilon\in\mathrm{seq}_2(\mc{A})/\mathbb{E}_2$ and $p=p_{\varepsilon,\bar{q}}$. 
Suppose for the sake of contradiction that $p\mid a+x$. Then for some $n \in \mathbb{N}$ there are $s_i\notin\Qp$ and $\bar{w}^i\in\varepsilon$ pairwise distinct  such that $p^{-1}(a+x)=\sum_{i<n}s_i(q_1w_1^i+q_2w_2^i)\mod {\mathbb{Z}_{(p)}\mc{A}\cap G}$.  Let $W=\{w\in\bar{w}^i:\bar{w}^i\in\varepsilon,i<n\}$ be the set of elements appearing among these pairs.

We will get a contradiction by showing that for $i<n$, there are  $\tilde{s}_i\notin\Qp$ and $r_1,r_2\in \Znz$  such that  $p^{-1}(r_1a+r_2x)=\sum_{i<n}\tilde{s}_i(q_1w_1^i+q_2w_2^i)\mod{\mathbb{Z}_{(p)}\mc{A}\cap G}$ and $p\nmid r_1$, $p \nmid r_2$, and $p \nmid r_1-r_2$. Then we find:
\begin{align*}
    p^{-1}r_1(a+x)-p^{-1}(r_1a+r_2x)&=r_1\sum_{i<n}s_i(q_1w_1^i+q_2w_2^i)-\sum_{i<n}\tilde{s}_i(q_1w_1^i+q_2w_2^i)\mod \mathbb{Z}_{(p)}\mc{A}\cap G\\
    p^{-1}(r_1-r_2)x&=\sum_{i<n}(r_1s_i-\tilde{s}_i)(q_1w_1^i+q_2w_2^i)\mod \mathbb{Z}_{(p)}\mc{A}\cap G
\end{align*}
As $r_1-r_2\not\equiv0\mod{p}$, we know that $|\mathrm{supp}_p(p^{-1}(r_1-r_2)x)|=1$. However, $|\mathrm{supp}_p(\sum_{i<n}(r_1s_i-\tilde{s}_i)(q_1w_1^i+q_2w_2^i))|\neq1$. If for all $i<n$, $r_1s_i-\tilde{s}_i \in \Qp$, then $|\mathrm{supp}_p(\sum_{i<n}(r_1s_i-\tilde{s}_i)(q_1w_1^i+q_2w_2^i))|=0$. If not, then we may remove any terms with coefficients in $\Qp$ while keeping the same $p$-support and obtain a non-trivial sum
\[ \sum_{\substack{i<n\\r_1s_i-\tilde{s}_i \notin \Qp}}(r_1s_i-\tilde{s}_i)(q_1w_1^i+q_2w_2^i).\]
If there is only one term left in the sum, then as $p \nmid q_1$ and $p \nmid q_2$, the $p$-support has size 2. Otherwise, by applying Lemma \ref{lem:family-a-modulo-p}, we find that the $p$-support has size at least 2. This would be a contradiction, so we conclude that $p\nmid (a+x)$. (One should think of this argument as showing that there is essentially a unique way, modulo $p$, of writing a non-trivial linear combination of $a$ and $x$ as a linear combination of $q_1 w_1^i + q_2 w_2^i$.)

To find such $r_1, r_2$ we consider the following graph. The vertices of the graph will be the elements of $W\subset \mc{A}$, and for $w_1,w_2\in W$, if $(w_1,w_2)$ is the corresponding ordered pair in the list $\bar{w}^i$, there will be an edge between $w_1$ and $w_2$. This graph has no cycles. If it did, we could consider the subgraph formed by the cycle, and take the set of tuples $\bar{w}^i$ represented by edges in the cycle. This  would form a subset of $\varepsilon$, but every vertex in the cycle would appear exactly twice. But part (3) of Theorem \ref{thm:eq-structure} tells us that for any subset of tuples in $\varepsilon$, at least two elements of the tuples should be unique, which cannot occur in a cycle. Thus this graph is a forest.

We will now argue that this forest has only two leaves, $a$ and $x$. In the sum $\sum_{i<n}s_i(q_1w_1^i+q_2w_2^i)$, any $w$ which is a leaf appears only a single time and so cannot cancel with any other term, and must be in the $p$-support. But $p^{-1}(a+x) = \sum_{i<n}s_i(q_1w_1^i+q_2w_2^i) \mod{\mathbb{Z}_{(p)}\mc{A}\cap G}$ and so the $p$-support is $\{a,x\}$. Thus $a$ and $x$ are the only possible leaves, and a finite forest must have at least two leaves per connected component. For example, the following graph is impossible.
\begin{align*}
    \xymatrix{&c\ar@{.}|-{\object@{/}}[dr]\ar@{.}|-{\object@{/}}[d]\ar@{.}|-{\object@{/}}[drr]\ar@{.}|-{\object@{/}}[drrr]&&&\\a \ar[ur]\ar[r]& b& d\ar[r]\ar[l]\ar[dr]& e \ar[r]&x\\
    &&&f\ar@{.}|-{\object@{/}}[ur]\ar@{.}|-{\object@{/}}[u]\ar@{.}|-{\object@{/}}[ull]\ar@{.}|-{\object@{/}}[ull]\ar@{.}|-{\object@{/}}[ulll]&}
\end{align*}
Here, $c$ and $f$ would not be able to cancel with any other term in the sum, and so we would have $|\mathrm{supp}_p(\sum_{i<n}s_i(q_1w_1^i+q_2w_2^i))|>2$. So the graph is simply a chain from $a$ to $x$.

Write out all the edges from left to right, each as its corresponding ordered tuple $\bar{w}^i$, beginning with the tuple containing $a$ and ending with the tuple containing $x$. Re-enumerate the tuples $\bar{w}^i$ based on their position in the chain, so the leftmost tuple (the one containing $a$) is $\bar{w}^0 = (w_1^0,w_2^0)$, and the rightmost tuple (the one containing $x$) is $\bar{w}^{n-1} = (w_1^{n-1},w_2^{n-1})$.  Note that these tuples are ordered, and, for example, $\bar{w}^0$ might be either $(a,w_2^0)$ or $(w_1^0,a)$. By part~(3) of Theorem~\ref{thm:eq-structure}, it cannot be that both $(w_1^i,w_2^i)$ and $(w_2^i,w_1^i)$ occur among the tuples $\bar{w}^j$, $j<n$.

We will use an inductive argument to show that at any stage $1\leq i\leq n$, there is $k \in \mathbb{Z}$ such that we can combine the tuples $\bar{w}^j = (w_1^j,w_2^j)$ for $j<i$ with some coefficients $s^*_j\in\Znz$ such that $p \nmid s^*_j$ for $j<i$, for which  \[a+(-1)^{k-1}q_1^{-k}q_2^kw=\sum_{j<i}s^*_j(q_1w_1^j+q_2w_2^j) \mod{p}\]
where $w \in W$ is the element in the chain which is $i$ steps from $a$. That is, $w$ is either $w_1^{i-1}$ or $w_2^{i-1}$, whichever is rightmost in the chain. When we write an equation like this modulo $p$, we mean an equality of formal linear combinations with the coefficients taken modulo $p$.

For the base case, suppose $i=1$. We will consider $\bar{w}^0=(a,w)$ and $\bar{w}^0=(w,a)$ as separate cases.
If $\bar{w}^0=(a,w)$, then we can multiply $q_1a+q_2w$ by $q_1^{-1}$ to get $a+q_1^{-1}q_2w$.  Then, let  $s^*_0\in\Znz$ be such that $s^*_0\equiv q_1^{-1}\mod{p}$. We have $a+q_1^{-1}q_2 w = s_0^* (q_1 a + q_2 w) \mod{p}$ (so $k = 1$). Similarly, if $\bar{w}^0=(w,a)$, then we can multiply $q_2a+q_1w$ by $q_2^{-1}$ to get $a+q_2^{-1}q_1w$. In this case,  let $s^*_0\in\Znz$ such that $s^*_0\equiv q_2^{-1}\mod{p}$. We have $a+q_1 q_2^{-1} w = s_0^* (q_1 w + q_2 a) \mod{p}$ (so $k = -1$). Note that for both choices of $s^*_0$, $p\nmid s^*_0$.

Now, let $2\leq i\leq n$ and suppose that for $j<i-1$ there are $s^*_j\in\Znz$ such that $p\nmid s^*_j$ and, for some $k\in\mathbb{Z}$, $a+(-1)^{k-1}q_1^{-k}q_2^ky=\sum_{j<i-1}s^*_j(q_1w_1^j+q_2w_2^j)\mod{p}$, where $y\in W$ is at the rightmost end of the chain of $\bar{w}^j$, $j<i-1$. Once again, we must treat the cases $\bar{w}^{i-1}=(y,w)$ and $\bar{w}^{i-1}=(w,y)$ for some $w\in W$ separately. If $\bar{w}^{i-1}=(y,w)$, then $a+(-1)^{k-1}q_1^{-k}q_2^ky+(-1)^kq_1^{-(k+1)}q_2^k(q_1y+q_2w)=a+(-1)^kq_1^{-(k+1)}q_2^{k+1}w$. Then, we can pick $s_{i-1}^*\in \Znz$ such that $s_{i-1}^*\equiv (-1)^kq_1^{-(k+1)}q_2^{k}\mod{p}$. We get
\[
a+(-1)^kq_1^{-(k+1)}q_2^{k+1}w=\sum_{j<i}s_j^*(q_1w_1^j+q_2w_2^j)\mod{p},
\]
giving the required form with $k+1$ in place of $k$. If $\bar{w}^{i-1}=(w,y)$, then $a+(-1)^{k-1}q_1^{-k}q_2^ky+(-1)^kq_1^{-k}q_2^{k-1}(q_1w+q_2y)=a+(-1)^kq_1^{-(k-1)}q_2^{k-1}w$, so we choose $s_{i-1}^*\in\Znz$ such that $s_{i-1}^*\equiv (-1)^kq_1^{-k}q_2^{k-1}\mod{p}$. We get
\[
a+(-1)^kq_1^{-(k-1)}q_2^{k-1}w=\sum_{j<i}s_j^*(q_1w_1^j+q_2w_2^j)\mod{p},
\]
giving the required form with $k-1$ in place of $k$. Note that in either case, $p \nmid s_{i-1}^*$.

Therefore, by induction, for every $1\leq i\leq n$ there is $k \in \mathbb{Z}$ and $s^*_j\in\Znz$ such that $p\nmid s^*_j$ and \[ a+(-1)^{k-1}q_1^{-k}q_2^kw=\sum_{j<i}s^*_j(q_1w_1^j+q_2w_2^j) \mod{p}.\]
In particular, at stage $i=n$ there are $s^*_j\in\Znz$ such that $p\nmid s^*_j$ and \[a+(-1)^{k-1}q_1^{-k}q_2^kx=\sum_{j<n}s^*_j(q_1w_1^j+q_2w_2^j) \mod{p}.\]
Let $r_1=1$ and $r_2 \in \mathbb{Z}$ be such that $r_2 \equiv (-1)^{k-1}q_1^{-k}q_2^k \mod p$. Then
\[r_1 a+ r_2 x=\sum_{j<n}s^*_j(q_1w_1^j+q_2w_2^j) \mod{p}.\]
Since $p \nmid s_j^*$, letting $\tilde{s}_j=p^{-1}s_j^*$ we have $\tilde{s}_j \notin\Qp$ for every $j<n$. The difference between the two sides in the equation above lies in $p\bigoplus_{y\in \mc{A}}\mathbb Zy$. Thus
\[p^{-1}(r_1 a+ r_2 x)=\sum_{j<n}\tilde{s}_j(q_1w_1^j+q_2w_2^j)\mod{\mathbb{Z}_{(p)}\mc{A}\cap G}.\]
By our choice of $p$ there does not exist a $k\in\mathbb{Z}$ such that $(-q_1^{-1}q_2)^k\equiv -1\mod{p}$, thus $(-1)^{k-1}q_1^{-k}q_2^k\not\equiv 1\mod{p}$ for any $k\in\mathbb{Z}$ (recall Remark \ref{rem:p-Zsigmondy}). So $p\nmid r_1$, $p \nmid r_2$, and $p\nmid r_1-r_2$ as desired.
\end{proof}

\begin{lemma}
    Fix a basis element $a \in \mc{A}$ of $G$. There is a computable $\Pi_1$ formula $\varphi(a,x)$ such that $G \models \varphi(a,b)$ if and only if $b \in G$ is a basis element of $G$.
\end{lemma}
\begin{proof}
      Let $P_1=\{p_{\varepsilon,\bar{q}} :\varepsilon\in\mathrm{seq}_n(\mc{A})/\mathbb{E}_n,n>1\}$ and $P_2=\{p_{\varepsilon,\bar{q}} :\varepsilon\in\mathrm{seq}_2(\mc{A})/\mathbb{E}_2 \text{ and } \bar{q} = (q_1,q_2) \text{ is such that } \gcd(q_1,q_2)=1, q_1>0, q_1\neq q_2\}$. The following formula $\varphi(a,x)$ is $\Pi_1$.
     \[
         \Biggl(\bigdoublewedge_{p\in P_2}\forall w(pw\neq x+a)\wedge\bigdoublewedge_{p_{*}\in P_1}\forall z(p_{*}z\neq x)\wedge\bigdoublewedge_{\substack{u,v\in\Znz\\u\neq v}}(ux\neq va)\Biggr)\vee(x=a).
     \]
   We claim that for $x\in G$, $\varphi(a,x)$ is true if and only if $x\in \mc{A}$ (that is, if $x$ is a basis element). 

   Let $x\in \mc{A}$. If $x=a$ then we are done. If $x\neq a$, then the divisibility conditions follow from Lemma \ref{lem:x is not 2x} and Lemma \ref{lem:x is not a sum}. The final conjunction follows from the rational linear independence of $a$ and $x$.

   Let $b\in G$ and suppose  $\varphi(a,b)$ is true. We want to show that $b\in \mc{A}$. If $b=a$ we are done, so suppose $b\neq a$. The $P_1$ clause implies that $b\neq0$, and the final conjunction then implies that $b\notin\mathbb{Q}a$. As $b\in G$, for some $n \in \mathbb{N}$, for $i<n$ there are $q_i\in \Qnz$ and distinct $x_i\in \mc{A}$ such that $b=\sum_{i<n}q_ix_i$.
      
      First, we claim that $n=1$. As $\varphi(a,b)$ is true, if $p\mid b$ then $p\not\in P_1$. Let $q_* \in \Znz$ be the product of the denominators of the $q_i$. %Then for every $i<n$, $q_*q_i\in\Znz$ and for every prime $p$, $p\mid q_*$ implies $p\mid q_*q_i$.
      Let $\bar{q}'=(q_*q_i)_{i<n}$, $\varepsilon=(x_i)_{i<n}/\Es{}{n}$, and $p=p_{\varepsilon,\bar{q}'}$. Then $p^{-m}\sum_{i<n}q_i'x_i$ is a generator of $G$ for all $m \in \mathbb{N}$. In particular, $q_*b=\sum_{i<n}q_i'x_i\in G_{\divs p}$, and---since $G_{\divs p}$ is a pure subgroup of $G$---we know $b\in G_{\divs p}$. Thus $p\mid b$. So $n=1$, as otherwise $p=p_{\varepsilon,\bar{q}'}\in P_1$. In particular, we can write $b=qx$, for some $x\in \mc{A}$ and $q\in\Qnz$.

      Second, we claim that $q=1$. Since $b=qx$ and $b\notin\mathbb{Q}a$, we have $x\neq a$. We can write $q=\frac{q_2}{q_1}$, where $q_1,q_2\in \Znz$ such that $q_1>0$ and $\gcd(q_1,q_2)=1$. Suppose for the sake of contradiction that $q_1\neq q_2$. Let $\varepsilon=(a,x)/\Es{}{2}$ and $\bar{q}=(q_1,q_2)$. Then $p_{\varepsilon,\bar{q}}\mid (q_1a+q_2x)$ and $p_{\varepsilon,\bar{q}} \nmid q_1$, so $p_{\varepsilon,\bar{q}}\mid (a+\frac{q_2}{q_1}x)$. But this is a contradiction, as $p_{\varepsilon,\bar{q}}\in P_2$ and so should not divide $a+b$. Therefore $q_2=q_1$, and $q=\frac{q_2}{q_1}=1$, as needed.
\end{proof}

\subsection{Choosing the right family and verification}

\begin{lemma}\label{lem:X non computable enumerates F}
     For a total computable function $f(n,s)$ consider the family $\mc{F}_f$ consisting of the sets $\{n\} \oplus F$ such that  if $\lim _{s\rightarrow \infty}f(n,s)=f(n)\downarrow$, then $F$ is finite and $F\neq W_{f(n)}$, and otherwise $F=\mathbb{N}$. If $X$ is non-computable then $X$ enumerates $\mc{F}_f$.
     
\end{lemma}

The argument is similar to other proofs of the Slaman-Wehner theorem, but it must handle the case where $\lim_{s\rightarrow\infty}f(n,s)$ does not converge. If $\lim_{s\rightarrow\infty}f(n,s)$ does not converge, then the opponent that we are diagonalizing against will keep changing, so we may end up building a copy of $X$ if the set we are constructing is challenged infinitely many times. This is a problem, as then two different non-computable sets $X$, $Y$ will not necessarily construct the same family, and the proof would fail. Therefore, each time there is a swap, we will add a new element into the sets in our family. In this way, we can ensure that if $\lim_{s\rightarrow\infty}f(n,s)$ does not converge, then for any non-computable set $X$ we will always build a copy of $\mathbb{N}$, avoiding the issue.

\begin{proof}
    Let $X$ be non-computable, and fix a computable listing
    $(F_i)_{i\in\N}$ of all finite subsets of $\N$, with $F_0=\emptyset$.
    Uniformly in $n$, we construct an $X$-c.e.\ enumeration of
    $\mathcal{F}_{f,n}=\{F:\{n\}\oplus F\in\mathcal{F}_f\}$. Fix $n$.

    For $t\leq s\in\N$, let $m_{f,n}(t,s)\in\N$ be the number of times $f$ switches between $t$ and $s$. That is, let $m_{f,n}(t,s)=|\{u\in\{t,...,s-1\}:f(n,u)\neq f(n,u+1)\}|$.

    For each finite set $F_i$ and every $t\in\N$, define $U_{F_i,t}[s]$ as follows. If $s\leq t$, let $U_{F_i,t}[s]=F_i$. If $s+1>t$, first account for the number of switches between $t$ and $s+1$ by defining $\tilde{U}_{F_i,t}[s+1]=U_{F_i,t}[s]\cup\{u\in\N:u<m_{f,n}(t,s+1)\}$.
    
    Compare $\tilde{U}_{F_i,t}[s+1]$ with the stage approximation $W_{f(n,s+1)}[s+1]$. If $\tilde{U}_{F_i,t}[s+1]=W_{f(n,s+1)}[s+1]$, let $x=\min\bigl((X\oplus X^c)\setminus\tilde{U}_{F_i,t}[s+1]\bigr)$ and put $U_{F_i,t}[s+1]=\tilde{U}_{F_i,t}[s+1]\cup\{x\}$. Otherwise, let $U_{F_i,t}[s+1]=\tilde{U}_{F_i,t}[s+1]$.

    We claim that $\{\bigcup_{s\in\N}U_{F_i,t}[s]:i,t\in\N\}=\mathcal{F}_{f,n}$ $(*)$. We will treat the case where $f(n,s)$ swaps finitely many times and where $f(n,s)$ swaps infinitely many times separately.

    If $f(n,s)$ swaps finitely many times, then $\lim_{s\rightarrow\infty}f(n,s)=f(n)$ exists. In particular,  there is some $\tilde{t}$ such that $f(n,s)=f(n)$ for all $s>\tilde{t}$. In this case, we will show both the inclusions in $(*)$:
    \begin{itemize}
        \item[$\supseteq$:] Let $F\in\mathcal{F}_{f,n}$. Then $F$ is finite and $F\neq W_{f(n)}$. Choose $t\geq\tilde{t}$ sufficiently large that $f(n,s)=f(n)$ for all $s\geq t$ and $W_{f(n)}[s]\neq F$ for all $s>t$. There are no switches after $t$, so $m_{f,n}(t,s)=0$ for all $s\geq t$, and no padding is added to $U_{F,t}$. Moreover, $U_{F,t}$ is never threatened after $t$. Thus $U_{F,t}[s]=F$ for every $s$, and $F=\bigcup_{s\in\N}U_{F,t}[s]$, as needed.

        \item[$\subseteq$:] Fix $F_i$ and $t$. Suppose that $\tilde{U}_{F_i,t}[s+1]=W_{f(n,s+1)}[s+1]$ for infinitely many $s$. Since $f(n,s)$ eventually stabilizes at $f(n)$, these equalities imply $W_{f(n)}=\bigcup_{s\in\N}U_{F_i,t}[s]$. At every such stage we add the least element of $X\oplus X^c$ not already present, so the union contains all of $X\oplus X^c$. The only other elements belong to $F_i$ or to the finite padding added at switches. Hence $W_{f(n)}$ is a finite extension of $X\oplus X^c$, contradicting that $X\oplus X^c$ is not computably enumerable. Thus there are only finitely many threats. Since there are also only finitely many switches, $\bigcup_{s\in\N}U_{F_i,t}[s]$ is finite. It cannot equal $W_{f(n)}$, since that would cause a further threat once both sets had stabilized. Therefore $\bigcup_{s\in\N}U_{F_i,t}[s]\in\mathcal{F}_{f,n}$.
    \end{itemize}

    If $f(n,s)$ swaps infinitely many times, then $\lim_{s\rightarrow\infty}f(n,s)$ does not exist. For every finite set $F_i$ and every $t\in\N$, the values $m_{f,n}(t,s)$ are unbounded as $s$ increases. Consequently, $\bigcup_{s>t}\{u\in\N:u<m_{f,n}(t,s)\}=\mathbb{N} \subseteq\bigcup_{s\in\N}U_{F_i,t}[s]$. Since the reverse inclusion holds by construction, every such union is $\mathbb{N}$, as required.

    The construction is uniform in $n,i,t$. This gives an $X$-c.e.\ enumeration
    of $\mathcal{F}_f$.
\end{proof}

\begin{definition}
    Let $(B_e,g_e)$ be a listing of the pairs of partial computable groups and their elements. 
    \begin{itemize}
        \item  Let $h(e,s)$ be the least element $x\in B_e$ among the first $s$ elements of $B_e$ such that $r_{2e}|x$ is witnessed by stage $s$ and $\varphi(g_e,x)$ looks like it is satisfied at stage $s$, or the $s$th element of $B_e$ if no such element is found.
        \item  Let $f:\N\times\N\rightarrow \N$ be the total computable function such that $f(e,s)$ is the index of the c.e. set  $\{n:r_{2n+1}|h(e,s)\}$. 
    \end{itemize}
\end{definition}

 Lemma \ref{lem:X non computable enumerates F} tells us that all non-computable sets $X$ enumerate $\mathcal{F}_f$. So in particular, by Lemma \ref{lem:enumeration isomorphism types}, all non-computable sets $X$ compute $G_{\mathcal{F}_f}$. We will show these are the only $X$ which compute $G_{\mathcal{F}_f}$.

\begin{lemma}
    The group $G_{\mathcal{F}_f}$ does not have a computable copy.
\end{lemma}
\begin{proof}
    Fix a basis element $a\in G_{\mathcal{F}_f}$. Suppose $G_{\mathcal{F}_f}$ does have a computable copy, then there is an $e\in\N$ such that $(G_{\mathcal{F}_f},a)\cong (B_e,g_e)$. 

    As $(G_{\mathcal{F}_f},a)\cong (B_e,g_e)$, $h(e,s)$ will converge to the least basis element $h(e)=x\in B_e$ divisible by $r_{2e}$. This must happen, as $B_e$ has at least one basis element $x$ divided by $r_{2e}$ (since $G_{\mathcal{F}_f}$ does, and $(G_{\mathcal{F}_f},a)\cong(B_e,g_e)$). Suppose that $x$ is the least basis element divisible by $r_{2e}$. Then after some finite number of stages we will have recognized that $x$ is divisible by $r_{2e}$ and also for any element less than $x$ which is not a basis element we will have recognized that it is not a basis element. From this point on, $h(e,s)$ will be $x$.

    Now, $\lim_{s\rightarrow\infty}f(e,s)=f(e)$  gives us the index of the set $W_{f(e)}=\{n:r_{2n+1}|h(e)\}$. Let $y\in \mc{A}$ correspond to $x\in B_e$ under the isomorphism, and let $W$ be the enumeration used to construct $G_{\mathcal{F}_f}$. By Lemma \ref{lem:r divisors}, $r_n\mid y$ if and only if $n\in \alpha_{\mathcal{F}_f}^W(y/\mathbb{X})$. Divisibility is preserved by the isomorphism, so $r_{2e}\mid y$ and $\{n:r_{2n+1}\mid y\}=W_{f(e)}$. Since $\alpha_{\mathcal{F}_f}^W(y/\mathbb{X})\in\mathcal{F}_f$, it follows that $\alpha_{\mathcal{F}_f}^W(y/\mathbb{X})=\{e\}\oplus W_{f(e)}$. But $\{e\}\oplus W_{f(e)}\notin\mathcal{F}_f$. This is a contradiction, so $(G_{\mathcal{F}_f},a)\not\cong (B_e,g_e)$. We conclude that $G_{\mathcal{F}_f}$ does not have a computable copy.
\end{proof}

Therefore, $G_{\mathcal{F}_f}$ is a torsion-free abelian group which has no computable copy, and has a copy computable from every non-computable set. This completes the proof of Theorem \ref{thm:slaman wehner}.

\section{Suggested open problems}\label{sec:open-problems}

In this section we suggest several open problems on the computable structure theory of torsion-free abelian groups which we suspect might be attacked using the methods of this paper.

We refer the reader to the survey \cite{Melnikov14} as a reference.

\subsection{$\Delta^0_n$ categoricity}

A computable structure $\mc{A}$ is \emph{$\Delta^0_{n}$-categorical} if there is a $\Delta^0_n$ isomorphism between any two computable copies. The following question was asked by Goncharov (see \cite{Melnikov14}).

\begin{question}[Goncharov]
    For each $n$, are there torsion-free abelian groups that are $\Delta^0_{n+1}$-categorical but not $\Delta^0_n$-categorical?
\end{question}

\noindent An affirmitive answer is obtained for $n = 1,\ldots,4$ and for odd $n$ in \cite{DowneyMelnikov13,DowneyMelnikov14,Melnikov14,MelnikovOptimalScott2018}. The remaining cases are for even $n \geq 6$ (and of course one could ask about infinite computable ordinals).

\begin{question}[Downey, Melnikov]
    For a fixed $n>2$, characterize the computable torsion-free abelian groups which are
    $\Delta^0_n$-categorical, and determine the     complexity of the corresponding index set.
\end{question}

At level $n=1$, a computable torsion-free abelian group is computably categorical if and
only if it has finite rank \cite{DowneyMelnikov13,Melnikov14}.
Downey and Melnikov \cite{DowneyMelnikov14} study certain cases of this question for computable completely decomposable groups.

\subsection{Jump degrees}

Let
\[
    \operatorname{DgSp}(\mc A)
      =\{\deg(\mc B):\mc B\cong\mc A\}
\]
be the degree spectrum of $\mc A$, the set of Turing degrees of copies of $\mc{A}$, and put
\[
    J_\alpha(\mc A)
      =\{\mathbf d^{(\alpha)}:
          \mathbf d\in\operatorname{DgSp}(\mc A)\}.
\]
The structure $\mc A$ has \emph{$\alpha$th jump degree}
$\mathbf d $ if $\mathbf d$ is the least member of
$J_\alpha(\mc A)$.  Equivalently, every copy $\mc B$ satisfies
$\mc B^{(\alpha)} \geq_T \mathbf{d}$, and some copy $\mc B$ satisfies
$ \mc B^{(\alpha)} \equiv_T \mathbf {d}$.  The jump degree is \emph{proper} if $\mc A$ has
no $\beta$-th jump degree for any $\beta<\alpha$.

Andersen, Kach, Melnikov, and Solomon \cite{AndersenKachMelnikovSolomon2012} proved that, for every computable
ordinal $\alpha$ and every degree
\[
    \mathbf d>\mathbf 0^{(\alpha)},
\]
there is a torsion-free abelian group with proper $\alpha$th jump degree
$\mathbf d$. $\mathbf{0}^{(\alpha)}$ is the least possible $\alpha$th jump degree.

\begin{question}[Andersen, Kach, Melnikov, Solomon]
    For every computable ordinal $\alpha\geq 3$, is there a torsion-free abelian group
    with proper $\alpha$th jump degree
    $\mathbf 0^{(\alpha)}$?
\end{question}

For $\alpha=0,1,2$, every possible proper jump degree is known to be
realizable.  Melnikov \cite{Melniokv17} later realized the minimum
$\mathbf 0^{(\alpha)}$ whenever
\[
    \alpha=\delta+2k+2,
\]
where $\delta=0$ or $\delta$ is a computable limit ordinal.
The problem remains open for odd successor
ordinals $\alpha\geq 3$ and for computable limit ordinals.

\subsection{Reduced parts}

For a torsion-free abelian group $G$, let
\[
    D(G)=\{x\in G:(\forall n\geq 1)(\exists y\in G)\;ny=x\}
\]
be its \emph{maximal divisible subgroup}.  Since divisible abelian groups are
injective, $D(G)$ is a direct summand: there is a reduced group $R$
such that
\[
    G=D(G)\oplus R.
\]
The choice of $R$ need not be canonical.  The \emph{canonical reduced part}
is the quotient
\[
    G_{\mathrm{red}}=G/D(G),
\]
which is isomorphic to every reduced complement $R$.

\begin{question}[Goncharov]
    If $G$ is a computable torsion-free abelian group, must its reduced part have a computable copy?
\end{question}

Khisamiev and Khisamiev claimed a negative answer in
\cite{Khisamiev}. However, Melnikov \cite{Melnikov14} writes that the proof is incomplete and therefore
the question remains open. 

\subsection{Decidably presentable torsion-free abelian groups}

A \emph{decidable copy} of a structure is a presentation whose complete elementary diagram is computable.

\begin{question}
    Characterize the torsion-free abelian groups which have a decidable copy, and determine the    complexity of the corresponding index set.
\end{question}

This question is stated by Harrison-Trainor in \cite{HT18} where it is shown via a $\Sigma^1_1$-completeness result that there is no such characterization for structures in general. Some results are known in specific cases. For the completely decomposable groups
\[
    G_S=\bigoplus_{n\in S} \mathbb Z[1/p_n],
\]
where $p_n$ is the $n$-th prime and $\mathbb Z[1/p_n]$ is the ring generated by $1/p_n$, Downey, Goncharov, Kach, Knight,
Kudinov, Melnikov, and Turetsky \cite{Downeyetal} proved
\[
    G_S\text{ has a decidable copy}
      \quad\Longleftrightarrow\quad
    S\in\Sigma^0_2,
\]
and
\[
    G_S\text{ has a computable copy}
      \quad\Longleftrightarrow\quad
    S\in\Sigma^0_3.
\]
There is a complete answer for homogeneous completely
decomposable groups. If
\[
    G=\bigoplus_{i \in \mathbb{N}}H
    \qquad\text{with }0\neq H\leq\mathbb Q,
\]
then $G$ has a decidable copy if and only if the Baer type of $H$
has a computable characteristic
\cite{BazhenovGoncharovMelnikov2020}.

\bibliography{References}
\bibliographystyle{plain}

\end{document}